\documentclass[10pt]{amsart}
\usepackage{amssymb,latexsym,amsfonts,multicol}
\usepackage{amsthm}
\usepackage{amscd}
\usepackage{amsmath,mathrsfs}
\usepackage{fontenc}
\usepackage{amsmath}
\usepackage{geometry}
\usepackage{graphicx}
\usepackage[all]{xy}
\usepackage[T1]{fontenc}
\usepackage{textcomp}
\newcommand{\om}{\omega}
\newcommand{\on}{\operatorname}
\newcommand{\der}{\on{der}}
\DeclareMathOperator{\pr}{pr}

\DeclareMathOperator{\colim}{colim}

\DeclareMathOperator{\coinv}{coinv}

\DeclareMathOperator{\ind}{ind}
\DeclareMathOperator{\For}{For}
\DeclareMathOperator{\Rep}{Rep}
\DeclareMathOperator{\Fl}{\mathfrak{Fl}}
\DeclareMathOperator{\Prest}{PreSt}

\DeclareMathOperator{\Free}{Free}

\DeclareMathOperator{\Vect}{Vect}

\DeclareMathOperator{\St}{St}

\DeclareMathOperator{\op}{op}

\DeclareMathOperator{\gr}{gr}

\DeclareMathOperator{\cL}{\mathcal{L}}

\DeclareMathOperator{\Spec}{Spec}

\DeclareMathOperator{\Id}{Id}

\DeclareMathOperator{\Sch}{Sch}
\DeclareMathOperator{\Alg}{AlgSp}
\DeclareMathOperator{\Aff}{Aff}
\DeclareMathOperator{\Hom}{Hom}

\DeclareMathOperator{\rs}{rs}
\DeclareMathOperator{\codim}{codim}

\DeclareMathOperator{\Fun}{Funct}

\DeclareMathOperator{\EExt}{\mathscr{E}\text{\kern -3pt {\calligra\large xt}}\,}
\DeclareMathOperator{\Ind}{Ind}
\DeclareMathOperator{\Ad}{Ad}

\DeclareMathOperator{\Aut}{Aut}

\newcommand{\cal}[1]{\mathcal{#1}}
\newcommand{\C}[1]{\cal#1}

\newcommand{\B}[1]{\mathbb#1}
\DeclareMathOperator{\Perv}{Perv}
\DeclareMathOperator{\Loc}{Loc}

\DeclareMathOperator{\perf}{perf}

\DeclareMathOperator{\Res}{Res}

\DeclareMathOperator{\val}{val}

\DeclareMathOperator{\pt}{pt}

\DeclareMathOperator{\Lie}{Lie}

\DeclareMathOperator{\res}{res}
\DeclareMathOperator{\red}{red}

\DeclareMathOperator{\ev}{ev}

\DeclareMathOperator{\triv}{triv}
\DeclareMathOperator{\RGamma}{R\Gamma}

\newtheorem{theorem}{Theorem}[subsection]

\newtheorem{Thm}[theorem]{Theorem}
\newtheorem{Prop}[theorem]{Proposition}
\newtheorem{Lem}[theorem]{Lemma}
\newtheorem{Cor}[theorem]{Corollary}

\newtheorem{Cl}[theorem]{Claim}

\theoremstyle{definition}

\newtheorem{Emp}[theorem]{}

\newtheorem{theorem'}{Theorem}[section]

\theoremstyle{definition}
\newtheorem{Emp'}[theorem']{}

\numberwithin{equation}{section}

\newcommand{\form}[1]{(\ref{Eq:#1})}

\newcommand{\rl}[1]{Lemma \ref{L:#1}}

\newcommand{\rp}[1]{Proposition \ref{P:#1}}

\newcommand{\re}[1]{\ref{E:#1}}
\newcommand{\rco}[1]{Corollary \ref{C:#1}}

\newcommand{\rt}[1] {Theorem \ref{T:#1}}

\newcommand{\cH}{\mathcal{H}}

\newcommand{\dt}{\delta}
\newcommand{\un}{\underline}

\newcommand{\clp}{\mathcal{L}^{+}}

\newcommand{\co}{\mathcal{O}}
\newcommand{\cO}{\mathcal{O}}

\newcommand{\cS}{\mathcal{S}}

\newcommand{\g}{\gamma}

\newcommand{\fC}{\frak{C}}

\newcommand{\fg}{\frak{g}}
\newcommand{\fc}{\frak{c}}

\newcommand{\isom}{\overset {\thicksim}{\to}}

\newcommand{\cY}{\mathcal{Y}}
\newcommand{\la}{\lambda}

\newcommand{\gm}{\mathbb{G}_{m}}

\newcommand{\fq}{\B{F}_q}

\newcommand{\kC}{\mathfrak{C}}

\newcommand{\al}{\alpha}
\newcommand{\br}{\textbf{r}}
\newcommand{\hra}{\hookrightarrow}
\newcommand{\sm}{\smallsetminus}
\newcommand{\cC}{\mathcal{C}}

\newcommand{\ov}{\overline}
\newcommand{\bql}{\overline{\mathbb{Q}}_{\ell}}

\newcommand{\kD}{\mathfrak{D}}
\newcommand{\kc}{\mathfrak{c}}
\newcommand{\kg}{\mathfrak{g}}

\newcommand{\NN}{\mathbb{N}}

\newcommand{\Gm}{\Gamma}

\newcommand{\qlbar}{\ov{\mathbb{Q}}_l}

\newcommand{\wt}{\widetilde}

\newcommand{\cX}{\mathcal{X}}
\newcommand{\cD}{\mathcal{D}}
\newcommand{\cZ}{\mathcal{Z}}
\newcommand{\cI}{\mathcal{I}}
\newcommand{\La}{\Lambda}
\newcommand{\cU}{\mathcal{U}}
\newcommand{\lra}{\longrightarrow}

\newcommand{\lla}{\longleftarrow}

\newcommand{\Affft}{\Aff^{\on{ft}}}
\newcommand{\Schft}{\Sch^{\on{ft}}}
\newcommand{\Algft}{\Alg^{\on{ft}}}

\def\dar[#1]{\ar@<2pt>[#1]\ar@<-2pt>[#1]}
\def\tar[#1]{\ar@<3pt>[#1]\ar@<-3pt>[#1]\ar@<0pt>[#1]}
\begin{document}

\title[Perversity of coinvariants of affine Springer sheaves]
{Perversity of coinvariants of affine Springer sheaves}

\author{Alexis Bouthier}

\address{Institut de Mathematiques de Jussieu\\
4 place Jussieu, 75005 Paris, France}
\email{alexis.bouthier@imj-prg.fr}

\author{David Kazhdan}

\author{Yakov Varshavsky}
\address{Einstein Institute of Mathematics\\
Edmond J. Safra Campus\\
The Hebrew University of Jerusalem\\
Givat Ram, Jerusalem, 9190401, Israel}
\email{kazhdan@math.huji.ac.il, yakov.varshavsky@mail.huji.ac.il}

\date{\today}

\thanks{The research was partially supported by ERC grant 101142781. The research of Y.V. was partially supported by the
ISF grant 2091/21.}

\begin{abstract}
Using techniques of \cite{BKV}, we construct a perverse $t$-structure on the $\infty$-category of $\ell$-adic $\cL G$-equivariant sheaves on the regular-semisimple bounded locus of the loop group $\cL G$ and prove that the derived $\tau$-coinvariants of affine Grothendieck--Springer sheaves are perverse. Our main new ingredient is a theorem of Yun \cite{Yun} on compatibility of actions.
\end{abstract}

\maketitle

\centerline{To G\'erard Laumon with deep respect}

\tableofcontents

\section*{Introduction}

\begin{Emp'}
{\bf The finite-dimensional case.} (a) Let $G$ be a connected reductive group over an algebraically closed field $k$, and let $T\subseteq B$ and $W$ be a maximal torus, a Borel subgroup and the Weyl group of $G$, respectively. Consider the diagram
\[
[G/G]\overset{\ov{\frak{p}}^{\on{fin}}}{\lla} [B/B]\overset{\pr}{\lra} T,
\]
where $[G/G]$ and $[B/B]$ denote quotient stacks corresponding to the adjoint actions, morphism $\ov{\frak{p}}^{\on{fin}}$ (called the {\em Grothendieck--Springer resolution}) is induced by the inclusion $B\hra G$, and morphism $\pr$ is induced by the projection $B\to B/R_u(B)\simeq T$.

\smallskip

(b) To every local system $\cL$ on $T$, we associate the {\em Grothendieck--Springer sheaf}
\[
\C{S}_{\cL}^{\on{fin}}:=(\ov{\frak{p}}^{\on{fin}})_!\pr^*(\cL)[\dim G]\in\C{D}([G/G]).
\]
Since $[B/B]$ is a smooth stack, while $\ov{\frak{p}}^{\on{fin}}$ is a small morphism, the sheaf  $\C{S}_{\cL}^{\on{fin}}$ is perverse and is the intermediate extension of its restriction to the regular semisimple locus $[G^{\rs}/G]$.

\smallskip

(c) Moreover, if $\cL$ is $W$-equivariant, then $\C{S}_{\cL}^{\on{fin}}$ is equipped with a $W$-action. In this case, for every representation $\tau\in\on{Rep}_{\qlbar}(W)$ one can form its $\tau$-isotypical component
\[
\C{S}_{\cL,\tau}^{\on{fin}}:=\tau\otimes_{\qlbar[W]} \C{S}_{\cL}^{\on{fin}}\in\C{D}([G/G]),
\]
which is again a perverse sheaf, and is the intermediate extension of its restriction to $[G^{\rs}/G]$. Furthermore, $\C{S}_{\cL,\tau}^{\on{fin}}$ is irreducible, if $\cL$ and $\tau$ are irreducible. 
\end{Emp'}

The goal of this work is to prove analogs of these results for loop groups, using constructions and results from \cite{BKV}.

\begin{Emp'}
{\bf The affine case.} (a) Let $\clp (G)$ and $\cL G$ be the arc group and the loop  group of $G$, respectively,
let $\ev:\clp(G)\rightarrow G$ be the projection (see Section~\re{looparc}), and let $I:=\ev^{-1}(B)\subseteq \clp(G)$ be the Iwahori subgroup scheme.

\smallskip

(b) Let $\fC\subseteq \cL\kg$ be the locus of ``bounded elements''. More precisely, we define $\fC\subseteq \cL\kg$ to be the preimage $\fC:=(\cL\chi)^{-1}(\clp(\fc))$, where $\fc:=\Spec k[G]^G$ is the Chevalley space of $\kg$, $\cL\chi:\cL G\to\cL\fc$ is the morphism between 
loop  spaces induced by the projection $\chi:G\to\fc$, and $\clp(\fc)\subseteq\cL\fc$ is the arc space of $\fc$.  Consider the diagram
\[
[\fC/\cL G]\overset{\ov{\frak{p}}}{\lla} [I/I]\overset{\pr}{\lra} T,
\]
where $[\fC/\cL G]$ and $[I/I]$ denote quotient stacks corresponding to the adjoint actions, morphism $\ov{\frak{p}}$ (called the {\em affine Grothendieck--Springer resolution}) is induced by the inclusion $I\hra \fC\subseteq \cL G$, and morphism $\pr:[I/I]\to T$ is induced by the projection $I\overset{\ev}{\lra} B\overset{\pr}{\to} T$.

\smallskip

(c) Let $\ell$ be a prime number different from the characteristic of $k$. Every $\infty$-stack $\cX$ over $k$ gives rise to a stable $\infty$-category $\C{D}(\C{X})$ of $\ell$-adic sheaves on $\C{X}$, and every morphism $f:\cX\to\cY$ of $\infty$-stacks gives rise to a pullback functor $f^!:\cD(\cY)\to \cD(\cX)$ (see Section~\re{sheaves}). In particular, to every $\infty$-stack $\cX$ one can associate a dualizing sheaf $\om_{\cX}\in\cD(\cX)$, defined to be the $!$-pullback of $\qlbar\in\cD(\pt)$.

\smallskip

(d) As in the Lie algebra case, the projection $\ov{\frak{p}}$ is ind-fp-proper, where ``fp'' stands for {\em finitely-presented} (see Section~\re{fibr}(b)), therefore the pullback $\ov{\frak{p}}^!:\cD([\fC/\cL G])\to \cD([I/I])$ has a left adjoint $\ov{\frak{p}}_!$ (see \cite[Proposition~5.3.7]{BKV}). For every local system $\cL$ on $T$, we set
\[
\C{S}_{\cL}:=\ov{\frak{p}}_!\pr^!(\om_T\otimes\cL)\in \cD([\fC/\cL G]),
\]
and call it the {\em affine Grothendieck--Springer sheaf}.

\smallskip

(e) We denote by $\fC_{\bullet}\subseteq\fC$ the open ind-subscheme such that $\fC_{\bullet}(k)=\fC(k)\cap G^{\rs}(k((t)))$, set $I_{\bullet}:=I\cap \fC_{\bullet}$, and let $\ov{\frak{p}}_{\bullet}:[I_{\bullet}/I]\to[\fC_{\bullet}/\cL G]$ and $\C{S}_{\cL,\bullet}\in \cD([\fC_{\bullet}/\cL G])$ be the restrictions of $\ov{\frak{p}}$ and let $\C{S}_{\cL}$, respectively.
\end{Emp'}

\begin{Emp'} \label{E:mainresults}
{\bf Main results.} (a) Assume that the derived group $G^{\der}$ of $G$ is simply connected and that the order of $W$ is prime to the characteristic of $k$. In this case, we show that the stable $\infty$-category $\cD([\fC_{\bullet}/\cL G])$ is equipped with
a natural perverse $t$-structure.

\smallskip

(b) Assume further that the characteristic of $k$ is either zero  or greater than $2h$, where $h$ is the Coxeter number of $G$.\footnote{In \cite{BKV} we forgot to impose this condition (see footnote after the formulation of Theorem~\ref{flat1} below).} In this case we show that for every local system $\cL$ on $T$, the sheaf $\C{S}_{\cL,\bullet}$ is perverse, and is the intermediate extension of its restriction to the locus
$[\fC_{\leq 0}/\cL G]$ of bounded elements with regular semisimple reduction. Moreover, we show that  if the local system $\C{L}$ is $W$-equivariant, then $\C{S}_{\cL,\bullet}$ is equipped with a natural action of the extended affine Weyl group $\wt{W}$ of $G$.
In this case, for every representation $\tau\in\on{Rep}_{\qlbar}(\wt{W})$, one can form the $\tau$-isotypical component
\[
\C{S}_{\cL,\bullet,\tau}:=\tau\otimes^L_{\qlbar[\wt{W}]} \C{S}_{\cL,\bullet}\in\C{D}([\fC_{\bullet}/\cL G]).
\]

\smallskip

(c) Furthermore, we show that $\C{S}_{\cL,\bullet}$ is {\em $\wt{W}$-constructible}, that is, the quotient stack $[\fC_{\bullet}/\cL G]$ has a {\em constructible stratification} such that $!$-restriction of  $\C{S}_{\cL,\bullet}$ to each stratum is a local system, whose fibers are perfect complexes of $\qlbar[\wt{W}]$-modules. Therefore, for every finite-dimensional  representation $\tau\in\on{Rep}_{\qlbar}(\wt{W})$, the $\tau$-isotypical component $\C{S}_{\cL,\bullet,\tau}$ is constructible.

\smallskip

(d) Finally, the main result of this work asserts that for every representation $\tau\in\on{Rep}_{\qlbar}(\wt{W})$, the $\tau$-isotypical component $\C{S}_{\cL,\bullet,\tau}$ is perverse.

\end{Emp'}

\begin{Emp'}
{\bf Remarks.} (a) Note that while the Lie algebra analogs of parts (a)-(c) were known before, the Lie algebra analog of part~(d) is new. 

\smallskip

(b) Unlike $\C{S}_{\cL,\bullet}$, perverse sheaf $\C{S}_{\cL,\bullet,\tau}$ is not always the intermediate extension of its restriction to
$[\fC_{\leq 0}/\cL G]$. Moreover, $\C{S}_{\cL,\bullet,\tau}$ is not always irreducible, when $\C{L}$ and $\tau$ are irreducible.
\end{Emp'}

\begin{Emp'}
{\bf Affine character sheaves.} (a) Perverse sheaves $\C{S}_{\cL,\tau}^{\on{fin}}$ provide examples of Lusztig's character sheaves, which play an important role in the representation theory of finite groups of Lie type.

\smallskip

(b) Likewise we expect that perverse sheaves $\C{S}_{\cL,\bullet,\tau}$ provide examples of {\em affine character sheaves}, whose general definition is currently unknown, but which are expected to play a similarly important role in the representation theory of $p$-adic groups.
\end{Emp'}

\begin{Emp'}
{\bf Outline of proofs.} (a) Since $G^{\der}$ is assumed to be simply connected, the Chevalley space $\fc$ is smooth. Then, mimicking the Lie algebra case, considered in \cite{BKV}, we show that the stack $[\fC_{\bullet}/\cL G]$ is {\em placidly stratified}, that is, has a constructible stratification $\{[\fC_{w,\br}/\cL G]_{\red}\}_{w,\br}$ such that each stratum is a {\em placid stack}. Therefore each  $\cD([\fC_{w,\br}/\cL G]_{\red})$ is equipped with a canonical ($!$-adapted) perverse $t$-structure, and the perverse structure of $\cD([\fC_{\bullet}/\cL G])$ is constructed by gluing.

\smallskip

(b) As in the Lie algebra case, to prove the assertion of Section~\re{mainresults}(b),
we show that the affine Grothendieck--Springer fibration  $\ov{\frak{p}}_{\bullet}:[I_{\bullet}/I]\to[\fC_{\bullet}/\cL G]$ is small.
Moreover, we deduce the smallness from codimension formula for Goresky--Kottwitz--MacPherson strata, dimension formula for affine Springer fibers
$\Fl_{\g}$ and flatness of projection $v_n:I_n\to \cL_n(\fc)$ between truncated arc spaces. Finally, we deduce the flatness of $v_n$ from the corresponding result for Lie algebras using Jordan decomposition and identifying the unipotent locus of the group with the nilpotent locus of
the Lie algebra. 

\smallskip

(c) To show the assertion of Section~\re{mainresults}(c), we prove the finiteness properties of the Grothendieck--Springer fibration $\ov{\frak{p}}_{\bullet}$, mimicking the Lie algebra case.

\smallskip

(d) Since $\C{S}_{\cL,\bullet}$ is perverse and the functor of $\tau$-coinvariants is right $t$-exact, to show the perversity of $\C{S}_{\cL,\bullet,\tau}$ it suffices to show that $\C{S}_{\cL,\bullet,\tau}\in {}^p\cD^{\geq 0}([\fC_{\bullet}/\cL G])$. Moreover, using Section~\re{mainresults}(c), we show that it suffices to prove that for every $\g\in\fC_{\bullet}(k)$, we have an inclusion
\begin{equation} \label{Eq:intr}
R\Gm_c(\Fl_{\g},\om_{\cL})_{\tau}\in D^{\geq -d_{\g}},
\end{equation}
where $\Fl_{\g}$ denotes the affine Springer fiber,  $R\Gm_c(\Fl_{\g},\om_{\cL})_{\tau}$ denotes the $\tau$-coinvariants of the cohomology with compact support of the pullback of $\om_T\otimes\cL$, and $d_{\g}$ is an explicit integer.

Finally, we deduce the inclusion \form{intr} from the group analog \cite[Theorem~2.3.4]{BV} of a theorem of Yun \cite{Yun} on the compatibility of actions.
\end{Emp'}

\begin{Emp'}
{\bf Remark.} Though we knew for a long time that the perversity of $\C{S}_{\cL,\bullet,\tau}$ follows from inclusion \form{intr} and that inclusion \form{intr} should follow from the theorem of Yun, we were able to complete the proof of inclusion \form{intr}
only recently. The missing step was Proposition~\ref{coh3}, and it was motivated by a nice trick explained to us by Zhiwei Yun.
\end{Emp'}

\begin{Emp'}
{\bf Relation to local Langlands conjectures and stability.} 
The local Langlands conjecture predicts that the set of isomorphism classes of smooth irreducible representations of $G(\fq((t)))$ has a natural partition into finite subsets $\Pi_{\la}$, called $L$-packets.

Moreover, it is expected that the linear span $\on{Span}\{\chi_{\pi}\}_{\pi\in\Pi_{\la}}$ of characters of $\pi\in\Pi_{\la}$ has another basis $\{\chi_{\la}^{\kappa}\}_{\kappa}$ such that each $\chi_{\la}^{\kappa}$ is ``$\C{E}_{\la,\kappa}$-stable''. Furthermore, it is expected that each $\chi_{\la}^{\kappa}$ is obtained from a perverse sheaf on $[\cL G/\cL G]$ by the ``sheaf-function correspondence''.

Using results of \cite{BV}, the perversity of $\cS_{\cL,\bullet,\tau}$ implies a version of the above conjecture for $L$-packets
of cuspidal Deligne--Lusztig representations, introduced in \cite{KV}.
\end{Emp'}

\begin{Emp'}
{\bf Plan of the paper.} The paper is organized as follows:

\smallskip

In the Section~1 we carry out geometric preliminaries:
First, in Section 1.1, we recall various notions from~\cite{BKV}, including placid $\infty$-stacks, placidly stratified $\infty$-stacks
and small morphisms. Next, in Section 1.2, we introduce $\Gm$-constructible and essentially constructible sheaves on $\infty$-stacks
and study their properties. We also introduce a subclass of placid $\infty$-stacks, which we call {\em admissible}.
Then, in Section 1.3, we recall a construction of perverse $t$-structures on placidly stratified $\infty$-stacks, introduced and studied
in \cite{BKV}, and give several criteria we need later. Finally, in Section 1.4, we show several simple properties of quasi-coherent sheaves, which play a central role in the proof of our main theorem.

\smallskip

In Section~2, we prove group analogs of some of the results of \cite{BKV}: First, in Section 2.1, we show flatness of the Chevalley map for
truncated arc spaces, deducing it from its analog for Lie algebras. Next, in Section 2.2, we introduce a GKM-stratification and
give a proof of a formula for the codimension of strata, which is much shorter than the original one. Finally, in Section~2.3, we study basic properties of the affine Grothendieck--Springer fibration, essentially mimicking the corresponding assertions for Lie algebras.

\smallskip

In Section~3, we introduce affine Grothendieck--Springer sheaves $\cS_{\cL,\bullet}$ and study their properties.
First, in Section 3.1, we show that each $\cS_{\cL,\bullet}$ is perverse, equipped with a natural $\wt{W}$-action and
that the induced action on homologies of affine Springer fibers coincides with the action constructed by Lusztig. Finally, in Section 3.2,
we show the main result of this work asserting that sheaves of $\tau$-coinvariants $\cS_{\cL,\bullet,\tau}$  are perverse.
For completeness, we also show the corresponding assertion for Lie algebras, whose proof is almost identical.

\end{Emp'}

\begin{Emp'}
{\bf Acknowledgments.} We thank Zhiwei Yun who explained to us a nice simple trick, which was crucial for the proof of our main theorem.

It is our pleasure and honor to dedicate this paper to G\'erard Laumon. Over the years, A. B. has greatly benefited from G\'erard's support, enthusiasm and insight. Through his work, his students, and his personal qualities, G\'erard has developed a wide network of knowledge and friendship that will last for many years to come. As used to say Gr\'egoire de Nysse: ``Celui qui s'\'el\`eve, il va de commencement en commencement, par des commencements qui n'ont jamais de fin''. Merci G\'erard.
\end{Emp'}

\section{Geometric preliminaries.}

\subsection{Placid stacks and small morphisms}

In this section we recall certain construction from~\cite{BKV}.

\begin{Emp}
{\bf Infinity-stacks.} Let $k$ be an algebraically closed field.

\smallskip

(a) Let $\Sch_k,\Aff_k$ and $\Alg_k$ be the categories of schemes, affine schemes and algebraic spaces over $k$, respectively, let $\Schft_k,\Affft_k$ and $\Algft_k$ be the subcategories of schemes, affine schemes and algebraic spaces of finite type over $k$, and let $\frak{S}$ be the $\infty$-category of spaces, which are often referred as $\infty$-groupoids.

\smallskip

(b) We denote by $\Prest_k$ the $\infty$-category of $\frak{S}$-valued presheaves on $\Aff_k$, that is,
of functors of $\infty$-categories $\Aff_k^{\op}\to\frak{S}$, usually called {\em $\infty$-prestacks} over $k$.
We denote by $\St_k\subseteq \Prest_k$ the full $\infty$-subcategory of sheaves in the
\'etale topology $\Aff_k^{\op}\to\frak{S}$, called {\em $\infty$-stacks}.

\smallskip

(c) We call a morphism $\cX\to\cY$ of $\infty$-stacks a {\em covering}, if it is a surjective map of sheaves. Explicitly, this means that for every
morphism $Y\to\C{Y}$ with $Y\in\Aff_k$ there exists an \'etale covering  $X\to Y$ of (affine) schemes such that the composition $X\to Y\to\C{Y}$ has a lift
$X\to\C{X}$.

\smallskip

(d) Generalizing classical construction for (affine) schemes, to every $\infty$-stack $\C{X}$ one can associate its reduction $\C{X}_{\red}$
(see~\cite[Section~1.4]{BKV}). 
As in \cite[Section~1.5.1]{BKV}, we say that a morphism between $\infty$-stacks $f:\C{X}\to\C{Y}$ is a {\em topological equivalence}, if the induced morphism $\C{X}_{\red}\to\C{Y}_{\red}$ between reductions is an isomorphism.\footnote{It is more natural to require that the induced map $\C{X}_{\perf}\to\C{Y}_{\perf}$ is an isomorphism, but this more general notion is not needed for our purposes.}

\smallskip

(e) As in \cite[Section~2.4.1]{BKV}, for every $\infty$-stack $\C{X}$ and an $\infty$-substack $\C{Y}\subseteq\C{X}$, one can form the complementary $\infty$-substack $\C{X}\sm\C{Y}\subseteq\C{X}$.

\smallskip

(f) For an $\infty$-stack $\C{X}$ and a point $x\in\C{X}(k)$, we denote by $\iota_x$ the corresponding morphism $\pt:=\Spec k\to \cX$.

\end{Emp}

\begin{Emp} \label{E:classes}
{\bf Classes of morphisms.}

\smallskip

(a) Let $(P)$ be a class of morphisms $f:\C{X}\to Y$ with $\cX\in\St_k$ and $Y\in\Aff_k$, stable under base change. We say that a morphism between $\infty$-stacks {\em belongs to $(P)$}, if its pullback to any affine scheme belong to $(P)$.

\smallskip

(b) Using construction of part~(a), we can talk about representable morphisms (corresponding to the class of all morphisms $X\to Y$ with $X\in\Alg_k$),
representable (locally)-fp morphisms (where {\em fp} stands for {\em finitely-presented}), and fp-open/(locally) closed embeddings between $\infty$-stacks.

\smallskip

(c) We say that a morphism $f:\C{X}\to Y$ with $\cX\in\St_k$ and $Y\in\Aff_k$ is {\em ind-fp-proper}, if $X$ has a presentation as a filtered colimit $\C{X}\simeq\colim_{\al}X_{\al}$ such that $X_{\al}\in \Alg_k$ for all $\al$, each  $X_{\al}\to Y$ is fp-proper, and each transition map
$X_{\al}\to X_{\beta}$ is an fp-closed embedding. Moreover, using the construction of part~(a) we can talk about ind-fp-proper morphisms
between $\infty$-stacks.

\smallskip

(d) We say that an $\infty$-substack $\cY\subseteq\cX$ is {\em topologically fp-(locally) closed},
if for every morphism $X\to\cX$ with $X\in \Aff_k$ there exists an fp-(locally) closed subscheme $Y\subseteq X$ and an isomorphism
$Y_{\red}\simeq (\cY\times_{\cX} X)_{\red}$ over $X$.

\end{Emp}

\begin{Emp} \label{E:consstr}
{\bf Constructible stratifications} (compare \cite[Section~2.4.5]{BKV}).

\smallskip

Let $\cX$ be an $\infty$-stack, let $\{\cX_{\al}\}_{\al\in\cI}$ be a collection of non-empty topologically fp-locally closed reduced $\infty$-substacks of $\cX$ with $\cX_{\al}\cap\cX_{\beta}=\emptyset$, and denote by $\eta_{\al}:\cX_{\al}\hra\cX$ the inclusion morphisms.

\smallskip

(a) We say that the collection $\{\cX_{\al}\}_{\al\in\cI}$ forms a {\em finite constructible stratification} of $\cX$, if $\cI$ is finite and there exists an ordering $\al_1<\dots<\al_n$ of $\cI$ and an increasing sequence of fp-open $\infty$-substacks $\emptyset=\cX_0\subseteq\dots\subseteq\cX_n=\cX$ of $\cX$ such that for every $i=1,\dots, n$ we have $\cX_{\al_i}\subseteq\cX_i\sm\cX_{i-1}$ and the embedding $\cX_{\al_i}\hra\cX_i\sm\cX_{i-1}$ is a topological equivalence.

\smallskip

(b) For an $\infty$-substack $\cY$ of $\cX$, we say that $\cY$ is $\{\cX_{\al}\}_{\al\in\cI}$-{\em adapted}, if for every $\al\in \cI$, we have  either $\cY\cap\cX_{\al}=\emptyset$ or $\cX_{\al}\subseteq\cY$. In which case, we set $\cI_{\cY}:=\{\al\in\cI~\vert~\cX_{\al}\cap\cY\neq\emptyset\}$.

\smallskip

(c) We say that $\{\cX_{\al}\}_{\al\in\cI}$ forms a {\em bounded constructible stratification} of $\cX$, if there exists a presentation $\cX\simeq\colim_{U} \cX_U$ as a filtered colimit such that each $\cX_U\subseteq\cX$ is an fp-open $\{\cX_{\al}\}_{\al\in\cI}$-adapted $\infty$-substack and the collection $\{\cX_{\al}\}_{\al\in\cI_{\C{X}_U}}$ forms a finite constructible stratification
of $\cX_U$.

\smallskip

(d) Let $f:\cY\to\cX$ be a morphism of $\infty$-stacks. Then a finite/bounded constructible stratification $\{\cX_{\al}\}_{\al\in\cI}$ of $\cX$ induces a corresponding stratification $\{\cY_{\al}:=f^{-1}(\cX_{\al})_{\red}\}_{\al\in\cI,\cY_{\al}\neq\emptyset}$
of $\cY$.
\end{Emp}



\begin{Emp} \label{E:glplsch}
{\bf Placid $\infty$-stacks and smooth morphisms.} 

\smallskip

(a) We say that a $k$-scheme $X$ {\em has a placid presentation}, if $X$ has a presentation as a filtered limit
$X\simeq\lim_{\al} X_{\al}$ such that $X_{\al}\in \Schft_k$ for all $\al$ and all transition maps $X_{\beta}\to X_{\al}$ are smooth and affine.
Such a presentation will be called {\em placid}.

\smallskip

(b)  We call a morphism $f:X\to Y$ of $k$-schemes {\em strongly pro-smooth}, if $X$ has a presentation as a filtered limit
$X\simeq\lim_{\al} X_{\al}$ over $Y$, where all morphisms $X_{\al}\to Y$ are smooth and finitely presented, while all projections $X_{\beta}\to X_{\al}$ are smooth, finitely presented and affine.

\smallskip

(c) We call a $k$-scheme $X$ {\em placid}, if it has an \'etale covering by schemes admitting placid presentations.
We call a morphism $f:X\to Y$ of placid $k$-schemes {\em smooth}, if locally in
the \'etale topology it is a strongly pro-smooth morphism of schemes.\footnote{As in~\cite{BKV}, our smooth morphisms are not assumed to be locally finitely presented.}

\smallskip

(d) More generally, following \cite[Section~1.3.1]{BKV}, we define a collection of {\em placid $\infty$-stacks} and a collection {\em smooth morphisms} between placid $\infty$-stacks, containing placid schemes and smooth morphisms from part~(c).
By definition, for every placid $\infty$-stack $\C{X}$, there is a smooth covering $X\to\C{X}$ from a placid scheme $X$. 

\smallskip

(e) A placid $\infty$-stack $\cX$ will be called {\em smooth}, if the structure morphism $\cX\to \pt$ is smooth.

\smallskip
\end{Emp}

\begin{Emp} \label{E:basic example}
{\bf Examples.}

\smallskip

(a) Let $G$ be a group scheme over $k$, whose neutral connected component is strongly pro-smooth, acting on a placid scheme $X$.
Then the quotient stack $\C{X}=[X/G]$ is placid, and the projection $X\to\C{X}$ is a smooth covering (see \cite[Section~1.3.9]{BKV}).
Moreover, all placid $\infty$-stacks appearing in this work are of this form.

\smallskip

(b) It follows from \cite[Lemma~1.3.6(c)]{BKV} that if $f:\cX\to\cY$ is a locally fp-morphism between placid $\infty$-stacks such that
$\cY$ is placid, then $\cX$ is placid. Also, by \cite[Corollary~1.4.5(b)]{BKV}, in this case the reduction $\cX_{\red}$ is placid and the embedding
$\cX_{\red}\to\cX$ is fp-closed.

\smallskip

(c) Using results of part~(b) one deduces that if $\cY$ is a placid $\infty$ stack and $\cX\subseteq\cY$ is a topologically fp-(locally) closed reduced $\infty$-substack, then $\cX\subseteq\cY$ is fp-(locally) closed and $\cX$ is placid.
\end{Emp}

\begin{Emp} \label{E:topsp}
{\bf Underlying topological space} (compare \cite[Section~2.2.1]{BKV}).

\smallskip

(a) Generalizing the classical notion for schemes, to every $\infty$-stack $\C{X}$, one associates the underlying topological space $[\C{X}]$ \label{N:[x]} such that

\smallskip

\quad\quad$\bullet$ the underlying set is defined to be the set of equivalent classes of pairs $(K,[z])$, where $K/k$ is a field extension,
$[z]\in \pi_0(\C{X}(K))$, and $(K',[z'])\sim (K'',[z''])$, if there exist field embeddings
$K'\hra K$ and $K''\hra K$ such that $[z']$ and $[z'']$ have the same image in $\pi_0(\C{X}(K))$.

\smallskip

\quad\quad$\bullet$ a subset $U\subseteq [\cX]$ is open, if $U=[\cU]$ for some  open $\infty$-substack $\cU\subseteq\cX$.

\smallskip

(b) Every morphism $f:\cX\to \cY$ of $\infty$-stacks induces a continuous map $[f]:[\cX]\to[\cY]$
of topological spaces. We call a morphism $f:\cX\to\cY$ of $\infty$-stacks {\em open}\label{I:open morphism}, if the induced map $[f]$ is open. We call $f$ {\em universally open}\label{I:universally open morphism},
if every pullback $\cX\times_{\cY}\cZ\to\cZ$ of $f$ is open.

\smallskip

(c) To simplify the notation, we will often denote the topological space $[\C{X}]$ by $\C{X}$ and the map $[f]$ by $f$.
\end{Emp}




\begin{Emp} \label{E:dimfn}
{\bf Dimension function.}

\smallskip

(a) For  every $X\in\Schft_k$ and $x\in X$, we denote by $\dim_x(X)$ the maximum of dimensions of irreducible components of $X$, containing $x$.

\smallskip

(b) Following \cite[Lemmas~2.2.4--2.2.5]{BKV}, to every locally fp-representable morphism $f:\C{X}\to\C{Y}$ between placid $\infty$-stacks
one can associate a dimension function $\un{\dim}_f:[\C{X}]\to\B{Z}$. Namely, it is uniquely characterised by the following properties:

\smallskip

\quad\quad (i) if $\C{X},\C{Y}\in \Schft_k$, then we have $\un{\dim}_f(x)=\dim_x(\cX)-\dim_{f(x)}(\cY)$ for every $x\in[\cX]$;

\smallskip

\quad\quad (ii) For every Cartesian diagram of placid $\infty$-stacks
\[
\begin{CD}
\C{X}' @>f'>> \C{Y}'\\
@V h VV @VV g V\\
\C{X} @>f>> \C{Y}
\end{CD}
\]
such that $g$ and $h$ are smooth and every $x'\in[\cX']$, we have an equality $\un{\dim}_{f'}(x')=\un{\dim}_{f}(h(x))$;

\smallskip

\quad\quad (iii) for every \'etale schematic morphism $g:\cZ\to\cX$ between placid $\infty$-stacks and every $z\in[\cZ]$, we have an equality
$\un{\dim}_{f\circ g}(z)=\un{\dim}_{f}(g(z))$.
\end{Emp}

\begin{Emp} \label{E:eqdim}
{\bf Equidimensional morphisms} (compare \cite[Section~2.2.6]{BKV}).

\smallskip

(a) A locally fp-representable morphism $f:\cX\to\cY$ of placid $\infty$-stacks is called

\smallskip

\quad\quad $\bullet$ {\em weakly equidimensional} (of relative dimension $d$), if the dimension function $\un{\dim}_f:[\cX]\to\B{Z}$ from
Section~\re{dimfn} is locally constant (constant with value $d$);

\smallskip

\quad\quad $\bullet$ {\em equidimensional}, if it is weakly equidimensional and satisfies $\un{\dim}_f(x) =\dim_x f^{-1}(f(x))$ for
every $x\in[\cX]$;

\smallskip

\quad\quad $\bullet$  {\em uo-equidimensional}, if it weakly equidimensional and universally open (see Section~\re{topsp}(b)).

\smallskip

(b) By \cite[Corollary~2.3.6]{BKV}, a uo-equidimensional morphism is equidimensional.

\smallskip

(c) We say that an fp-locally closed $\infty$-substack $\cX\subseteq\cY$ of a placid $\infty$-stack $\cY$ is of {\em (pure) codimension $d$}
and write $\codim_{\cX}(\cY)=d$, if the inclusion $\cX\hra\cY$ is weakly equidimensional of relative dimension $-d$.
\end{Emp}


\begin{Emp} \label{E:plstr}
{\bf Placidly stratified $\infty$-stacks.}

\smallskip

(a) We say that an $\infty$-stack $\cY$ is {\em $\cI$-stratified}, if it is equipped with a
bounded constructible stratification $\{\cY_{\al}\}_{\al\in\cI}$ (see Section~\re{consstr}).

\smallskip

(b) We say that an $\cI$-stratified $\infty$-stack $(\cY, \{\cY_{\al}\}_{\al\in\cI})$ is
{\em placidly stratified}, if every $\cY_{\al}$ is a placid $\infty$-stack (see Section~\re{glplsch}(d)). \end{Emp}

\begin{Emp} \label{E:small}
{\bf Small morphisms} (compare \cite[Section~2.4.9]{BKV}).

\smallskip

(a) Let $(\cY, \{\cY_{\al}\}_{\al\in\cI})$ be a placidly stratified $\infty$-stack, $\cU\subseteq\cY$ an fp-open $\{\cY_{\al}\}_{\al\in\cI}$-adapted $\infty$-substack, $f:\cX\to\cY$ a morphism of $\infty$-stacks, $\{\cX_{\al}\}_{\al}$ the induced bounded constructible stratification of $\cX$ (see Section~\re{consstr}(d)), and denote by $f_{\al}:\cX_{\al}\to\cY_{\al}$ the restriction
of $f$.

\smallskip

(b) In the situation of part~(a), assume that $\cX$ is placid. Then, by Section~\re{basic example}(c), each $\cX_{\al}\subseteq\cX$ is a placid
fp-locally closed $\infty$-substack.

\smallskip

(c) In the situation of part~(b), we say that $f$ is $\cU$-\textsl{small} if for every $\al\in\cI$ there exist integers  $b_{\al},\dt_{\al}\in\B{N}$ such that

\smallskip

\quad\quad (i) each $\cX_{\al}\subseteq\cX$ is of pure codimension $b_{\al}$;

\smallskip
	
\quad\quad (ii) each $f_{\al}$ is locally fp-representable and equidimensional of relative dimension $\delta_{\al}$;

\smallskip
	
\quad\quad (iii)	for every $\al\in\cI$, one has $\delta_{\al}\leq b_{\al}$, with a strict inequality for $\al\in\cI\sm\cI_{\cU}$.
\end{Emp}


\subsection{$\Gm$-constructible sheaves on $\infty$-stacks}
In this section we will introduced $\Gm$-constructible sheaves on $\infty$-stacks and show their basic properties.

\begin{Emp} \label{E:sheaves}
{\bf Sheaves on $\infty$-prestacks} (compare \cite[Sections~5.2.1, 5.2.2, 5.3.1]{BKV}).
Let $\ell$ be a prime, different from the characteristic of $k$.

\smallskip


(a) For an affine scheme $Y\in \Aff^{\on{ft}}_k$, we denote by $\cD_c(Y):=\cD_c(Y,\qlbar)$ the stable $\infty$-category
of constructible sheaves on $Y$.  
Then, for an affine scheme $X$ over $k$, we set $\cD_c(X):=\colim_{X\to Y}\cD_c(Y)$, where the colimit is taken over all morphisms
$X\to Y$ with $Y\in \Aff^{\on{ft}}_k$, and the transition maps are $!$-pullbacks. Next, we set $\cD(X):=\Ind\cD_c(X)$.

\smallskip

(b) For an $\infty$-prestack $\C{X}$ over $k$, we set $\cD_c(\C{X}):=\lim_{X\to \C{X}} \cD_c(X)$ and $\cD(\C{X}):=\lim_{X\to \C{X}} \cD(X)$, where the limits are taken over all morphisms $X\to\C{X}$ with $X$ affine, and the transition maps are $!$-pullbacks.  Notice that $\cD_c(\C{X})\subseteq \cD(\C{X})$ is a full $\infty$-subcategory, and we call objects of $\cD_c(\C{X})$ {\em constructible}.
By construction, the $\infty$-category $\cD(\C{X})$ is equipped with a $!$-tensor product.

\smallskip

(c) Notice that for every $Y\in \Schft_k$ or, more generally, $Y\in \Algft_k$ the $\infty$-category  $\cD_c(Y)$ is naturally identifies with  the $\infty$-category of constructible sheaves on $Y$.
\end{Emp}

We will need the following generalization of constructible sheaves.
\begin{Emp} \label{E:esscons}
{\bf Essentially constructible sheaves.}

\smallskip

(a) For an affine $k$-scheme $X$, we denote by $\cD_{\on{ess}-c}(X)\subseteq\cD(X)$ the smallest full $\infty$-subcategory, which contains
$\cD_c(X)$ and is closed under retracts, finite colimits and tensor products $-\otimes_{\bql}V$, where $V$ is a $\bql$-vector space.

\smallskip

(b) For a morphism $f:X\to Y$ of affine $k$-schemes, the pullback $f^{!}$ preserves the $\infty$-subcategories $\cD_{\on{ess}-c}(-)\subseteq\cD(-)$, so
 for an arbitrary $\infty$-prestack $\cX$, we can form a full $\infty$-subcategory
\[
\cD_{\on{ess}-c}(\cX):=\lim_{X\to \C{X}}\cD_{\on{ess}-c}(X)\subseteq \cD(\C{X}),
\]
and call objects of $\cD_{\on{ess}-c}(\cX)$ {\em essentially constructible}.
\smallskip

(c) By construction, for every morphism  $f:\cX\to \cY$ of $\infty$-prestacks, the pullback $f^!$ preserves the class of essentially constructible objects.
\end{Emp}

\begin{Lem} \label{L:ess-c}
Let $X\in\Schft_k$ and $K\in\cD(X)$. Then we have $K\in \cD_{\on{ess}-c}(X)$ if and only if there exists a finite constructible stratification $\{X_{\al}\}_{\al}$ of $X$ such that for each index $\al$,

\smallskip

\quad $\bullet$ only finitely many of cohomologies $\C{H}^i(\eta_{\al}^!(K))$'s are nonzero.

\smallskip

\quad$\bullet$ each $\C{H}^i(\eta_{\al}^!(K))$ is a finite extension of $L\otimes_{\bql}V$, where $L$ is an irreducible constructible local system
on $X$, and $V$ is a $\bql$-vector space.

\smallskip
Moreover, we can further assume that each $X_{\al}$ is connected smooth and affine, while each embedding $\eta_{\al}:X_{\al}\to X$ is weakly equidimensional.
\end{Lem}

\begin{proof}
Let $\cD'_{\on{ess}-c}(X)$ be the collection of all $K\in\cD(X)$ such that  there exists a finite constructible stratification $\{X_{\al}\}_{\al}$ of $X$ satisfying the conditions of the lemma.
Then $\cD'_{\on{ess}-c}(X)$ contains $\cD_c(X)$, closed under retracts, finite colimits and tensor products $-\otimes_{\bql}V$ and therefore contains $\cD_{\on{ess}-c}(X)$.

\smallskip

It remains to show that every $K\in\cD'_{\on{ess}-c}(X)$ belongs to $\cD_{\on{ess}-c}(X)$. Since $K$ is a finite extension of the
$(\eta_{\al})_*\eta_{\al}^!K$'s, it suffices to show that each  $(\eta_{\al})_*\eta_{\al}^!(K)$ is essentially constructible.
Next, since each $(\eta_{\al})_*$ preserve constructibility, finite colimits, retracts and commute with $-\otimes_{\qlbar}V$, it preserves essentially constructibility. So it suffices to show that each  $\eta_{\al}^!(K)$ is essentially constructible. But this follows from the fact that
each $\eta_{\al}^!(K)$ is a finite extension of its shifted cohomologies  $\cH^i(\eta_{\al}^!(K))[-i]$ and definition of $\cD'_{\on{ess}-c}(X)$.

\smallskip

Finally, to show the ``moreover'' assertion we argue similarly but use the fact that every constructible stratification $\{X'_{\beta}\}_{\beta}$ has a refinement $\{X_{\al}\}_{\al}$ such that each $X_{\al}$ is connected smooth and affine, while each embedding $\eta_{\al}:X_{\al}\to X$ is weakly equidimensional.
\end{proof}

\begin{Emp} \label{E:admstack}
{\bf Admissible $\infty$-stacks.}

\smallskip

(a) We call $Y\in\Affft_k$ {\em acyclic}, if the canonical morphism  $\bql\to\RGamma(Y,\bql)$ is an isomorphism.

\smallskip

(b) Following \cite[Section~1.1.2]{BeKV1}, we call a $k$-scheme $X$ {\em admissible}, if it admits a placid presentation $X\simeq\lim_{\al} X_{\al}$ such that all transition maps are smooth affine and have acyclic geometric fibers, and we will call such presentation {\em admissible}.

\smallskip

(c) It follows from \cite[Lemma~1.1.3]{BeKV1} that a presentation $X\simeq\lim_{\al} X_{\al}$ is admissible if and only if the  pullback $\pi_{\al}^{!}:\cD(X_{\al})\to\cD(X)$, corresponding to each projection $\pi_{\al}:X\to X_{\al}$, is fully faithful.

\smallskip

(d) We call a placid $\infty$-stack {\em admissible}, if it has a smooth covering by a disjoint union of admissible schemes.

\smallskip

(e) An argument of \cite[Lemma~1.3.6(b)]{BKV} shows that $f:X\to Y$ is an fp-morphism of $k$-schemes such that
$Y$ is an admissible, then $X$ is admissible.
\end{Emp}

\begin{Lem} \label{L:adm-essc}
Let $X$ be an affine $k$-scheme with an admissible presentation $X\simeq\lim_{\al}X_{\al}$. Then every $K\in \cD_{\on{ess}-c}(X)$
is a pullback of some $K_{\al}\in \cD_{\on{ess}-c}(X_{\al})$.
\end{Lem}

\begin{proof}
We denote by $\cD'_{\on{ess}-c}(X)\subseteq\cD_{\on{ess}-c}(X)$ the full $\infty$-subcategory, consisting of objects, which are pullbacks
of some $K_{\al}\in \cD_{\on{ess}-c}(X_{\al})$. By definition, $\cD'_{\on{ess}-c}(X)$  contains $\cD_{c}(X)$ and is closed under tensor
products $-\otimes_{\qlbar} V$. It remains to show that $\cD'_{\on{ess}-c}(X)$ is closed under finite colimits and retracts.
But this follows from the fact that the limit $X\simeq\lim_{\al}X_{\al}$ is filtered and each $\pi_{\al}^!:\cD(X_{\al})\to\cD(X)$ is fully faithful (see Section~\re{admstack}(c)).
\end{proof}

\begin{Emp} \label{E:catequiv}
{\bf Categories of $\Gm$-equivariant objects.}
\smallskip

(a) Let $\Gm$ be a (discrete) group, and let $B\Gm$ be the classifying space of $\Gm$. For an $\infty$-category $\cC$, we denote by
$\cC^{B\Gm}$ the $\infty$-category of functors $B\Gm\to\cC$, and call it the $\infty$-category of {\em $\Gm$-equivariant objects} in $\cC$. Explicitly, objects of  $\cC^{B\Gm}$ are objects of $\cC$ equipped with an action of $\Gm$.

\smallskip

(b) The assignment $\cC\mapsto \cC^{B\Gm}$ is functorial in $\Gm$. In particular, for every homomorphism of groups $\al:\Gm_1\to\Gm_2$, we have a restriction functor $\res_{\Gm_2}^{\Gm_1}=\al^*:\cC^{B\Gm_2}\to \cC^{B\Gm_1}$,
whose left adjoint we denote by $\ind_{\Gm_1}^{\Gm_2}=\al_!$ (if exists).
Notice that $\al_!$ always exists, if $\cC$ is {\em cocomplete}, that is, has all small colimits.

\smallskip

(c) An important particular case is when $\al$ is the projection $\Gm\to \{1\}$. In this case, $\al_!$ is the functor of coinvariants $\coinv_{\Gm}$. Another important particular case is when $\al$ is the embedding $\{1\}\to \Gm$. In this case, $\al^*$ is the forgetful functor $\For$, and $\al_!$ is the functor $\on{Free}$.

\smallskip

%

\smallskip

(d) The assignment $\C{C}\to\C{C}^{B\Gm}$ is functorial in $\C{C}$. In particular, every functor $f:\C{C}_1\to\C{C}_2$ of $\infty$-categories
naturally lifts to a functor  $f:\C{C}^{B\Gm}_1\to\C{C}^{B\Gm}_2$. Moreover, $f$ commutes with functors $\al^*$ from part~(b). Furthermore,
 $f$ commutes with functors $\al_!$, if $\C{C}_1$ and $\C{C}_2$ are cocomplete and $f$ preserves all small colimits.
\end{Emp}

\begin{Emp} \label{E:basicex}
{\bf Basic example.} Let $\C{C}$ be the derived $\infty$-category $\Vect$ of $\qlbar$-vector spaces. In this case, the $\infty$-category
$\C{C}^{B\Gm}$ is naturally identified with the derived $\infty$-category $\C{D}(\qlbar[\Gm])$ of (left) $\qlbar[\Gm]$-modules or, what is the same,
of representations of $\Gm$ over $\qlbar$.

\end{Emp}

\begin{Emp} \label{E:isot}
{\bf Functor of $\tau$-coinvariants.}

\smallskip

(a) Let $\C{C}$ be a $\qlbar$-tensored stable $\infty$-category. Then the tensor product functor
\[
\otimes_{\qlbar}:\Vect\times \C{C}\to \C{C}:(V,K)\mapsto V\otimes_{\qlbar}K
\]
defines a functor $\otimes_{\qlbar}:\Vect^{B\Gm}\times \C{C}^{B\Gm}\simeq(\Vect\times \C{C})^{B\Gm}\to \C{C}^{B\Gm}$.

\smallskip
(b) Assume in addition that $\C{C}$ is cocomplete. Then (by Sections~\re{catequiv}(c) and \re{basicex}) for every $\tau\in \Rep_{\qlbar}(\Gm)$ and $K\in\C{C}^{B\Gm}$, one can form $\tau$-coinvariants
\[
K_{\tau}=\coinv_{\tau}(K):=\coinv_{\Gm}(\tau\otimes_{\qlbar}K)\in \C{C}.
\]

\smallskip

(c) The assignment $\coinv_{\tau}$ is functorial in $\C{C}$. Namely, let $f:\C{C}_1\to\C{C}_2$ be a $\qlbar$-tensored functor between cocomplete
$\qlbar$-tensored stable $\infty$-categories. Then $f$ induces a functor $f:\C{C}^{B\Gm}_1\to\C{C}^{B\Gm}_2$, and for every
$K\in  \C{C}^{B\Gm}_1$ we have an identification $f(K)_{\tau}\simeq f(K_{\tau})$.

\smallskip

(d) Assume that $\C{C}=\Vect$. Then every $\tau\in\Rep_{\qlbar}(\Gm)$ can be viewed as a right $\qlbar[\Gm]$-module (using involution 
$\iota:\g\mapsto\g^{-1}$ of $\Gm$). Moreover, for every $V\in\Vect$, we have an identification $V_{\tau}\simeq\tau\otimes_{\qlbar[\Gm]}V$. Moreover, a similar formula holds for an arbitrary $\C{C}$.
\end{Emp}




\begin{Emp} \label{E:eqshv}
{\bf Equivariant sheaves.}

\smallskip

(a) Let $\cX$ be an $\infty$-stack, equipped with an action of a group $\infty$-stack $\C{G}$, and we denote by $p$ the projection $\C{X}\to[\C{X}/\C{G}]$. We set $\cD^{\C{G}}(\C{X}):=\cD([\C{X}/\C{G}])$, refer to objects of $\cD^{\C{G}}(\C{X})$ as {\em $\C{G}$-equivariant sheaves} on $\C{X}$ and view the pullback $p^!$ as the {\em forgetful functor}.

\smallskip

(b) Notice that if $\C{G}$ is a (discrete) group $\Gm$, then the projection
$p:\C{X}\to[\C{X}/\Gm]$ is ind-fp-proper (see \cite[Section~5.6.4]{BKV}), thus a pullback $p^!$ has a left adjoint $p_!$ commuting with base change (see \cite[Proposition~5.3.7]{BKV}).
\end{Emp}

\begin{Emp} \label{E:triv}
{\bf The case of a trivial action.} Let $\cX$ be an $\infty$-stack.

\smallskip

(a) For every group $\Gm$, we equip $\cX$ with the trivial action of $\Gm$. Then, using standard bar-resolution argument, there is a natural equivalence between the $\infty$-category $\cD(\cX)^{B\Gm}$ of $\Gm$-equivariant objects of
$\cD(X)$ and the $\infty$-category $\cD([\C{X}/\Gm])$ of $\Gm$-equivariant sheaves on $\cX$.

\smallskip

(b) A group homomorphism $\al:\Gm_1\to\Gm_2$ induces a morphism of $\infty$-stacks $\al:[\C{X}/\Gm_1]\to [\C{X}/\Gm_2]$. Then, under
the equivalence of part~(a), the restriction $\res_{\Gm_2}^{\Gm_1}:\cD(\C{X})^{B\Gm_2}\to\cD(\C{X})^{B\Gm_1}$ corresponds to the
pullback $\al^!:\cD([\C{X}/\Gm_2])\to \cD([\C{X}/\Gm_1])$, while the induction $\ind_{\Gm_2}^{\Gm_1}:\cD(\C{X})^{B\Gm_1}\to\cD(\C{X})^{B\Gm_2}$ corresponds to the pushforward $\al_!:\cD([\C{X}/\Gm_1])\to \cD([\C{X}/\Gm_2])$.

\smallskip

(c) In particular, the forgetful functor $\For:\cD(\cX)^{B\Gm}\to \cD(\cX)$ corresponds to the pullback $p^!:\cD([\C{X}/\Gm])\to \cD(\C{X})$, hence functor $\on{Free}$ (see Section~\re{catequiv}(c)) corresponds to the push-forward $p_!$, thus justifying the conventions of Section~\re{eqshv}.
\end{Emp}

\begin{Lem} \label{L:taucoinv}
Let $\C{X}$ be an $\infty$-stack equipped with an action of a group $\Gm$, and let $p$ be the projection $\C{X}\to\C{Y}:=[\C{X}/\Gm]$.
Then

\smallskip

(a) For every $K\in \cD(\C{Y})$, the object $p_!p^!K\in \cD(\cY)$ has a natural lift $\wt{K}\in \cD(\C{Y})^{B\Gm}$ and satisfies
\[
\coinv_{\Gm}(\wt{K})\simeq K\in\cD(\C{Y}).
\]

\smallskip

(b) Moreover, for every representation $\tau\in\Rep_{\qlbar}(\Gm)$, we have a natural isomorphism
\[
\coinv_{\tau}(\wt{K})\simeq A_{\tau}\overset{!}{\otimes}K,
\]
where $A_{\tau}=\pr^!(\tau)$, $\pr:[\cX/\Gm]\to [\pt/\Gm]$ is the canonical projection corresponding to the projection $\cX\to\pt$, and we identify $\tau$ with the corresponding object
of $\cD([\pt/\Gm])$.
\end{Lem}

\begin{proof}
(a) Projection $p$ gives rise to a Cartesian diagram
\[
\begin{CD}
\cX @>p>> \cY\\
@VpVV @VVpV\\
\cY=[\cX/\Gm] @>\ov{p}>> [\cY/\Gm],
\end{CD}
\]
where $\Gm$ acts on $\cY$ trivially, and $\ov{p}$ is induced by $p$. Then, by base change, we get a canonical isomorphism $p_!p^!K\simeq p^!(\ov{p}_!K)$, thus $\wt{K}:=\ov{p}_!K\in \C{D}([\cY/\Gm])\simeq\C{D}(\cY)^{B\Gm}$ is a lift of $p_!p^!K$.

Moreover, under the identification $\C{D}([\cY/\Gm])\simeq\C{D}(\cY)^{B\Gm}$, the functor $\coinv_{\Gm}$ corresponds to the push-forward
$\pi_!:\cD([\cY/\Gm])\to \cD(\cY)$, corresponding to the projection $\pi:[\cY/\Gm]\to\cY$.
Since $\pi\circ\ov{p}\simeq\Id$, the isomorphism $\coinv_{\Gm}(\wt{K})=\coinv_{\Gm}(\ov{p}_!K)\simeq K$ follows.

\smallskip

(b) By definitions, we have $\coinv_{\tau}(\wt{K})\simeq \coinv_{\Gm}(\tau\otimes_{\qlbar}\wt{K})$,
and $\tau\otimes_{\qlbar}\wt{K}\simeq\ov{\pr}^!(\tau)\overset{!}{\otimes}\wt{K}$, where
$\ov{\pr}:[\cY/\Gm]\to [\pt/\Gm]$ is the projection induced by the projection $\cY\to\pt$. Moreover, we have
\[
\ov{\pr}^!(\tau)\overset{!}{\otimes}\wt{K}=\ov{\pr}^!(\tau)\overset{!}{\otimes}\ov{p}_!(K)\simeq \ov{p}_!(\ov{p}^!\ov{\pr}^!(\tau))\overset{!}{\otimes}K)\simeq \ov{p}_!(\pr^!(\tau)\overset{!}{\otimes}K)=\ov{p}_!(A_{\tau}\overset{!}{\otimes}K)
\]
by the projection formula. Hence, as in part~(a), we conclude that
\[
\coinv_{\tau}(\wt{K})\simeq \coinv_{\Gm}(\ov{p}_!(A_{\tau}\overset{!}{\otimes}K))\simeq A_{\tau}\overset{!}{\otimes}K.
\]
\end{proof}

\begin{Emp} \label{E:Gmcons}
{\bf $\Gm$-constructible sheaves.} Assume that we are in the situation of Section~\re{triv}.

\smallskip

(a) For every $X\in\Aff_k$, we denote by $\cD_{\Gm,c}(X)\subseteq \cD^{\Gm}(X)$ the full $\infty$-subcategory of compact objects. Alternatively, $\cD_{\Gm,c}(X)\subseteq \cD^{\Gm}(X)$ can be described the smallest $\infty$-subcategory, which contains objects $\Free(K)$ for $K\in\cD_c(X)$ and is stable under finite colimits and retracts. Objects of $\cD_{\Gm,c}(X)$ will be called {\em $\Gm$-constructible}.
Notice that for every morphism $f:X\to Y$ of affine schemes, the pullback $f^!:\cD^{\Gm}(Y)\to \cD^{\Gm}(X)$ preserves
$\cD_{\Gm,c}(-)\subseteq\cD^{\Gm}(-)$.

\smallskip

(b) The assignment $\cC\mapsto \cC^{B\Gm}$ commutes with limits. Therefore for every $\infty$-stack $\C{X}$, the $\infty$-category $\cD^{\Gm}(\C{X})$ is naturally identified with a limit $\lim_{X\to \C{X}}\cD^{\Gm}(X)$, taken over maps from  affine schemes $X$. Hence, using part~(a), we define the full $\infty$-subcategory
\[
\cD_{\Gm,c}(\C{X}):=\lim_{X\to \C{X}}\cD_{\Gm,c}(X)\subseteq \cD(\C{X}).
\]
\end{Emp}

\begin{Emp} \label{E:example}
{\bf Examples.} (a) Let $\C{X}=\pt$. Then, by  Sections~\re{basicex} and \re{triv}(a), the $\infty$-category $\cD^{\Gm}(\pt)$ is identified
with the derived $\infty$-category $\cD(\qlbar[\Gm])$ of $\qlbar[\Gm]$-modules.
Under this identification, $\cD_{\Gm,c}(\pt)$ corresponds to the subcategory $\cD_{\perf}(\qlbar[\Gm])\subseteq \cD(\qlbar[\Gm])$ of perfect complexes.

\smallskip

(b) Note that if $A$ is a Noetherian ring of finite cohomological dimension, then an object $V\in \cD(A)$
is perfect if and only if the direct sum of cohomology groups $\bigoplus_i H^i(V)$ is a finitely generated $A$-module.
\end{Emp}

\begin{Emp} \label{E:func}
{\bf Functoriality properties.} Let $\cX$ be an $\infty$-stack.

\smallskip

(a) For every group homomorphism $\al:\Gm_1\to\Gm_2$, the induction functor $\al_!:\cD^{\Gm_1}(\C{X})\to  D^{\Gm_2}(\C{X})$ maps
$D_{\Gm_1,c}(\C{X})$ to $D_{\Gm_2,c}(\C{X})$. Indeed, since $\al_!$ commutes with $!$-pullbacks, we have to show the assertion when $\C{X}$ is an  affine scheme $X$. In this case, the assertion follows from the fact that the pushforward $\al_!: \cD^{\Gm_1}(X)\to  D^{\Gm_2}(X)$
preserves compact objects.

\smallskip

(b) The forgetful map $\For:\cD^{\Gm}(\C{X})\to \cD(\C{X})$ maps $\cD_{\Gm,c}(\C{X})$ to $\cD_{\on{ess}-c}(\C{X})$. Indeed, arguing as in part~(a), we can assume that $\C{X}$ is an affine scheme $X$. Thus, it suffices to show that for every $K\in D_c(X)$ we have
$\For(\Free(K))\in \cD_{\on{ess}-c}(X)$. But this follows from the isomorphism $\For(\Free(K))\simeq\qlbar[\Gm]\otimes_{\qlbar}K$.
\smallskip

\end{Emp}

\begin{Lem} \label{L:Gmconsstr}
Let $\cX$ be an $\infty$-stack, let $\Gm$ be a group, and let $K\in\cD^{\Gm}(\cX)$.

\smallskip

(a) Let $\{\cX_{\al}\}_{\al}$ be a bounded constructible stratification of $\cX$. Then we have $K\in\cD_{\Gm,c}(\cX)$ if and only if $\eta_{\al}^!(K)\in \cD_{\Gm,c}(\cX_{\al})$ for all $\al$.

\smallskip

(b) Let $f:\cY\to\cX$ be a covering of $\infty$-stacks. Then we have $K\in \cD_{\Gm,c}(\cX)$  if and only if $f^!K\in \cD_{\Gm,c}(\cY)$.

\smallskip
(c) If $K\in D_{\Gm,c}(\C{X})$, then for every $\tau\in \Rep_{\qlbar}(\Gm)$, we have $K_{\tau}\in \cD_{\on{ess}-c}(\C{X})$. Moreover, we have
$K_{\tau}\in \cD_{c}(\C{X})$, if $\tau$ is finite-dimensional.
\end{Lem}

\begin{proof}
(a) The ``only if'' assertion follows from the fact that the $!$-pullbacks preserve $\cD_{\Gm,c}$.
For the converse, choose a presentation $\C{X}=\colim_U\C{X}_U$ as in Section~\re{consstr}(c). Since every morphism $X\to\cX$ from
an affine scheme $X$ factors through some $\C{X}_U$, we conclude that $K\in\cD_{\Gm,c}(\cX)$ if and only if $K|_{\cX_U}\in\cD_{\Gm,c}(\cX_U)$ for all $U$. Thus, replacing $\C{X}$ and $K$ by $\cX_U$ and $K|_{\cX_U}$, we can assume that the stratification is finite.

Next, taking pullback to an affine scheme, we can assume that $X$ is an affine scheme. Since $K$ is a finite extension of the $(\eta_{\al})_*\eta_{\al}^!(K)$'s, it remains to show that functors $(\eta_{\al})_*$ preserve $\Gm$-constructibility.
If $X_{\al}$ is affine, the assertion follows from the fact that $(\eta_{\al})_*$ preserve compact objects. To show the general case,
choose a finite constructible stratification $\{X_{\beta}\}_{\beta}$ of $X_{\al}$ by affine schemes and
let $\eta_{\beta}$ be the composition $X_{\beta}\hra X_{\al}\hra X$. Then the assertion for $(\eta_{\al})_*$ follows from that for the $(\eta_{\beta})_*$'s.

\smallskip

(b) By construction, the $!$-pullbacks preserves $\cD_{\Gm,c}$. To show the converse assertion, notice that for every morphism
$X\to \C{X}$ from an affine scheme $X$ there exists an \'etale covering $Y\to X$  such that the composition $Y\to X\to\cX$ has a lift $Y\to\cY$. Thus, it suffices to show the assertion when $f:Y\to X$ is an \'etale covering of affine schemes. Then $f$ is finitely-presented,
so there exists a finite constructible stratification $\{X_{\al}\}_{\al}$ of $X$ such that each $Y_{\al}:=f^{-1}(X_{\al})\to X_{\al}$ is finite
\'etale.  Thus, by part~(a), we can assume that $f$ is finite \'etale. In this case, the assertion follows from the fact that $K$ is a
retract of $f_!f^!(K)$.

\smallskip

(c) Since $\coinv_{\tau}$ commutes with $!$-pullbacks, we can assume that $\cX$ is an affine scheme and
$K=\Free(K')$ for some $K'\in \cD_c(X)$. Using isomorphisms $\tau\otimes_{\qlbar}\Free(K')\simeq \Free(\qlbar^{\dim\tau}\otimes_{\qlbar}K)$ and
$\coinv_{\Gm}\circ\Free\simeq\Id$, we conclude that in this case we have $\coinv_{\tau}(K)\simeq\qlbar^{\dim\tau}\otimes_{\qlbar}K$,
from which the assertion follows.
\end{proof}

\begin{Lem} \label{L:La-cstr}
Let $f:X\to S$ a morphism of schemes, let $\Gm$ be a group acting on $X$ over $S$ such that $Y:=[X/\Gm]$ is an algebraic space, fp-proper over $S$. Then the pushforward $f_!:\cD(X)\to \cD(S)$ exists, and for every $K\in\cD(Y)$, the object $f_!(p^!K)\in\cD(S)$ has a natural lift to $\cD^{\Gm}(S)$. Moreover, this lift
belongs to $\cD_{\Gm,c}(S)$ if $K\in \cD_c(Y)$.
\end{Lem}

\begin{proof}
Morphism $f$ gives rise to a Cartesian diagram
\[
\begin{CD}
X @>f>> S\\
@VpVV @VVpV\\
Y=[X/\Gm] @>\ov{f}>> [S/\Gm],
\end{CD}
\]
where $\Gm$ acts on $S$ trivially. Moreover, morphisms $f$ and $\ov{f}$ decomposes as  $X\overset{p}{\to}Y\overset{g}{\to}S$ and
$Y\overset{g}{\to}S\overset{p}{\to}[S/\Gm]$, respectively, and morphism $g:Y\to S$ is fp-proper. By  \cite[Proposition~5.2.5]{BKV}, our assumption on $g$ implies that $g^!$ has a left adjoint $g_!$, from which the existence of $f_!$ and $\ov{f}_!$ follow.

\smallskip

Moreover, since $p$ is \'etale, we have an isomorphism
$f_!(p^!K)\simeq p^!(\ov{f}_!K)$, thus  $f_!(p^!K)\in\cD(S)$ has a natural lift $\ov{f}_!K\in\cD([S/\Gm])\simeq\cD^{\Gm}(S)$.
Finally, if $K\in \cD_c(Y)$, then $g_!K\in \cD_c(S)$, thus $\ov{f}_!K=p_!(g_!K)\in \cD_{\Gm,c}(S)$.
\end{proof}

\begin{Prop} \label{P:Gmcons}
Let $\cX$ be an admissible $\infty$-stack, let $\Gm$ be a group such that $\qlbar[\Gm]$ is Noetherian of finite cohomological dimension, and let   $K\in\cD^{\Gm}(\cX)$. Then we have $K\in\cD_{\Gm,c}(\cX)$ if and only if we have $\For(K)\in\cD_{\on{ess}-c}(\cX)$
and $\iota_x^!K\in\cD_{\Gm,c}(\pt)=\cD_{\perf}(\qlbar[\Gm])$ for every $x\in \cX(k)$.
\end{Prop}

\begin{proof}
By \rl{Gmconsstr}(b), we can assume that $\cX$ is an admissible affine scheme $X$.

\smallskip

Assume first $X$ is of finite type over $k$. If $K\in\cD_{\Gm,c}(X)$, then $\For(K)\in\cD_{\on{ess}-c}(X)$ by Section~\re{func}(b), and each $\iota_x^!K\in\cD_{\Gm,c}(\pt)$ because $\iota_x^!$ preserves $\Gm$-constructibility. Conversely, assume that $\ov{K}:=\For(K)\in\cD_{\on{ess}-c}(X)$ and $\iota_x^!K\in \cD_{\perf}(\qlbar[\Gm])$ for every $x\in X(k)$.

\smallskip

By \rl{Gmconsstr}(a), it suffices to show that there exists a finite constructible stratification $\{X_{\al}\}_{\al}$ of $X$ such that
$\eta_{\al}^!K\in \cD_{\Gm,c}(X_{\al})$ for all $\al$. Thus, by \rl{ess-c}, we can replace  $K$ by  $\eta_{\al}^!(K)$, thus assuming that $X$ is smooth connected, $\ov{K}\in \cD^b(X)$ and each $\cH^i(\ov{K})$ is a finite extension of the $L\otimes_{\qlbar}V$'s, where $L$ is an irreducible constructible local system on $X$ and $V$ is a $\qlbar$-vector space.

\smallskip

Since $\qlbar[\Gm]$ is Noetherian of finite cohomological dimension, we have $\iota_x^!K\in\cD_{\perf}(\qlbar[\Gm])$ if and only if 
$\bigoplus_i H^i(\iota_x^!K)$ is a finitely generated $\qlbar[\Gm]$-module. Since $K$ is a finite extension of its shifted cohomologies, and $\iota_x^![-2\dim X]\simeq \iota^*_x$ is exact, we can replace $K$ by $\bigoplus_i \C{H}^i(K)$, thus assuming that $\ov{K}=\For(K)$ is a finite extension of the $L\otimes_{\qlbar}V$'s, where $L$ is a simple (constructible) local system on $X$ and $V$ is a $\qlbar$-vector space.

\smallskip

In this case, $\ov{K}$ has a canonical finite increasing filtration $\{\ov{K}_{\leq i}\}_i$, defined by the property that $\ov{K}_{\leq i+1}/\ov{K}_{\leq i}$ is the largest semisimple local subsystem of $\For(K)/\For(K)_{\leq i}$. In particular, each  $\ov{K}_{\leq i}$ is $\Gm$-invariant, so filtration $\{\ov{K}_{\leq i}\}_i$ gives rise to a filtration $\{K_{\leq i}\}_i$ of $K$.

\smallskip

Replacing $K$ by its graded piece $\gr^i(K)$, we can assume that $\ov{K}$ is semisimple and has only many simple factors. In this case, $K$ decomposes as a finite direct sum $\bigoplus_{L}K_{L}$, where $L$ runs over the set of isomorphism classes of simple local systems $L$ on $X$, and  $K_{L}:=\Hom(L,\ov{K})\otimes_{\qlbar}L$. Now our assumption that each $H^i(\iota_x^!K)$ is a finitely generated $\qlbar[\Gm]$-module
implies that each $\Hom(L,\ov{K})$ is a finitely generated $\qlbar[\Gm]$-module. Thus $K_{L}\in\cD_{\Gm,c}(X)$ for all $L$, hence $K\in\cD_{\Gm,c}(X)$.

\smallskip

In the general case, choose an admissible presentation $X\simeq\lim_{\al}X_{\al}$. Then $\For(K)$ is a pullback of $\For(K)_{\al}\in\cD_{\on{ess}-c}(X_{\al})\subseteq \cD(X_{\al})$ for some $\al$ (by \rl{adm-essc}). Next, using the fact that the pullback $\pi_{\al}^!:\cD(X_{\al})\to\cD(X)$ is fully faithful, we conclude that the natural functor
\[
\cD^{\Gm}(X_{\al})\to\cD(X_{\al})\times_{\cD(X)}\cD^{\Gm}(X)
\]
is an equivalence.  Thus, $K$ is a pullback of a unique object $K_{\al}\in \cD^{\Gm}(X_{\al})$ such that
\[
\For(K_{\al})\simeq \For(K)_{\al}\in\cD_{\on{ess}-c}(X_{\al}).
\]

Since presentation $X\simeq\lim_{\al}X_{\al}$ is admissible, the projection $\pi_{\al}:X(k)\to X_{\al}(k)$ is surjective. Hence,
our assumption implies that $\iota_{x_{\al}}^!(K_{\al})\in\cD_{\Gm,c}(\pt)$ for every $x_{\al}\in X_{\al}(k)$. Then, we get $K_{\al}\in \cD_{\Gm,c}(X_{\al})$ (by the particular case shown above), hence $K=\pi_{\al}^!(K_{\al})\in\cD_{\Gm,c}(X)$.
\end{proof}


\subsection{Perverse $t$-structures}

\begin{Emp} \label{E:tstrdc}
{\bf Perverse $t$-structures on $\cD_c(X)$.} Following \cite[Section~6.2.3]{BKV}, for every $X\in\Schft_k$ we equip the $\infty$-category $\cD_c(X)$ with a canonical ($!$-adapted) perverse $t$-structure $({}^p\cD_c^{\leq 0}(X),{}^p\cD_c^{\geq 0}(X))$:

\smallskip

(a) If $X$ is equidimensional, we equip $\cD_c(X)$ with the $t$-structure $({}^p\cD_c^{\leq 0}(X),{}^p\cD_c^{\geq 0}(X))$,
obtained from the classical (middle-dimensional) perverse $t$-structure by shifting it by $\dim X$ to the left.

\smallskip

(b) For a general $X$, consider finite constructible stratification $X_i:=\{x\in X~\vert~\dim_x(X)=i\}$ (see \cite[Section~2.1.1]{BKV}).
Then each $X_i$ is equidimensional of dimension $i$, and we denote the inclusion $X_i\hra X$ by $\eta_i$. We define $({}^p\cD_c^{\leq 0}(X),{}^p\cD^{\geq 0}(X))$ to be the $t$-structure on $\cD_c(X)$, obtained from the $t$-structures $({}^p\cD_c^{\leq 0}(X_i),{}^p\cD_c^{\geq 0}(X_i))$ from part~(a) by gluing. Explicitly, we define
${}^p\cD_c^{\leq 0}(X)$ (resp. ${}^p\cD_c^{\geq 0}(X)$) to be the collection of all $K\in \cD_c(X)$ such that
$\eta_i^*(K)\in {}^p\cD_c^{\leq 0}(X_i)$ (resp. $\eta_i^!(K)\in {}^p\cD_c^{\geq 0}(X_i)$) for all $i$.

\smallskip

(c) By \cite[Lemma~6.2.5(c)]{BKV}, for every smooth morphism $f:X\to Y$, the pullback $f^!$ is $t$-exact.
\end{Emp}

\begin{Emp} \label{E:tstr}
{\bf Perverse $t$-structures on placid $\infty$-stacks.}

\smallskip

(a) Following \cite{BKV}, for every placid $\infty$-stack $\C{X}$ we equip the $\infty$-category $\cD(\cX)$
with a canonical perverse $t$-structure $({}^p\cD^{\leq 0}(\cX),{}^p\cD^{\geq 0}(\cX))$.
Namely, using \cite[Lemma~6.1.2(a) and Propositions~6.3.1, 6.3.3]{BKV} it can be uniquely characterised by the following properties:

\smallskip

\quad (i) when $\cX\in\Schft_k$, then $({}^p\cD^{\leq 0}(\cX),{}^p\cD^{\geq 0}(\cX))$ extends the $t$-structure $({}^p\cD_c^{\leq 0}(\cX),{}^p\cD_c^{\geq 0}(\cX))$ on $\cD_c(\cX)$ described in Section~\re{tstrdc}(b);

\smallskip

\quad (ii) for every smooth morphism $f:\cX\to\cY$ of placid $\infty$-stacks, the pullback $f^!$ is $t$-exact;

\smallskip

\quad (iii) the $t$-structure $({}^p\cD^{\leq 0}(\cX),{}^p\cD^{\geq 0}(\cX))$ is {\em compatible with filtered colimits}, that is,
the $\infty$-subcategory ${}^p\cD^{\geq 0}(\cX)\subseteq\cD(\cX)$ is closed under filtered colimits.

\smallskip

(b) By \cite[Section~5.3.1(e) and Lemma~6.1.2(d)]{BKV}, we conclude that for every smooth covering $f:\cX\to\cY$ of placid $\infty$-stacks
and $K\in \cD(\cX)$ we have $K\in{}^p\cD^{\geq 0}(\cX)$ if and only if $f^!K\in{}^p\cD^{\geq 0}(\cY)$.
\end{Emp}

\begin{Lem} \label{L:gluing}
Assume that $X\in\Schft_k$ is equipped with a finite constructible stratification $\{X_{\al}\}_{\al}$ such that each $X_{\al}\subseteq X$ is of pure codimension $d_{\al}$, let $\eta_{\al}:X_{\al}\hra X$ be the inclusion maps, and let $K\in\cD(X)$. Then we have $K\in \mathstrut^{p}\cD^{\leq 0}(X)$ (resp. $K\in \mathstrut^{p}\cD^{\geq 0}(X)$) if and only if $\eta_{\al}^*K\in {}^p\cD^{\leq -d_{\al}}(X_{\al})$ (resp. $\eta_{\al}^!K_{\al}\in {}^p\cD^{\geq -d_{\al}}(X_{\al})$) for all $\al$.
\end{Lem}

\begin{proof}
The ``only if'' assertion follows from the fact that functor $\eta_{\al}^*[-d_{\al}]$ is right $t$-exact, while functor $\eta_{\al}^![-d_{\al}]$ is left   $t$-exact (see \cite[Lemma~6.2.5(b)]{BKV}). By adjointness, functor $(\eta_{\al})_![d_{\al}]$ is right  $t$-exact, while functor $(\eta_{\al})_*[d_{\al}]$ is left  $t$-exact. From this the ``if'' assertion follows. Indeed, if $\eta_{\al}^!K\in {}^p\cD^{\geq -d_{\al}}(X_{\al})$ for all $\al$, then $(\eta_{\al})_*\eta_{\al}^!K\in {}^p\cD^{\geq 0}(X)$ for all $\al$. Since $K$ is a finite extension of the $(\eta_{\al})_*\eta_{\al}^!K$'s,
 we conclude that $K\in \mathstrut^{p}\cD^{\geq 0}(X)$. The assertion for $\mathstrut^{p}\cD^{\leq 0}(X)$ is similar.
\end{proof}

\begin{Prop}\label{pullcheck}
Let $\cX$ be an admissible $\infty$-stack (see Section~\re{admstack}), and we equip $\cD(\cX)$ with its canonical perverse $t$-structure
(see Section~\re{tstr}). Let $K\in \cD_{\on{ess}-c}(\cX)$ be such that for every  $x\in \cX(k)$, we have $\iota_x^{!}K\in \mathstrut^{p}\cD^{\geq 0}(\pt)$.
Then $K\in \mathstrut^{p}\cD^{\geq 0}(\cX)$.
\end{Prop}


\begin{proof}
By Section~\re{tstr}(b), we can assume that $\cX$ is an affine admissible $k$-scheme $X$. 

\smallskip

Assume first that $X$ is of finite type over $k$. By \rl{ess-c}, there exists a finite constructible stratification
$\{X_{\al}\}_{\al}$ of $X$ such that each $X_{\al}$ is smooth connected, each inclusion $\eta_{\al}:X_{\al}\hra X$ is weakly equidimensional of relative dimension $-d_{\al}$ and all cohomology sheaves $\cH^i(K_{\al})$ of all $K_{\al}:=\eta^!_{\al}K$ are local systems.

By \rl{gluing}, we have $K\in \mathstrut^{p}\cD^{\geq 0}(X)$ if and only if  $K_{\al}\in {}^p\cD^{\geq -d_{\al}}(X_{\al})$ for all $\al$.
In particular, it suffices to show that  $K_{\al}\in {}^p\cD^{\geq 0}(X_{\al})$ for all $\al$. Thus, replacing $X$ by $X_{\al}$ and $K$ by
$K_{\al}$, we can assume that $X$ is smooth connected and all cohomology sheaves $\cH^i(K)$ are local systems.
Then our assumption that $\iota_x^{!}K\simeq \iota_x^* K[2\dim X]\in \mathstrut^{p}\cD^{\geq 0}(\pt)$ implies that
$K\in \cD^{\geq -2\dim X}(X)$, where $\cD^{\geq -2\dim X}$ refers to the usual (rather than perverse) $t$-structure. Since $\cD^{\geq -2\dim X}(X)=\mathstrut^{p}\cD^{\geq 0}(X)$, we are done.

\smallskip

For a general $X$, choose an admissible presentation $X\simeq\lim_{\al} X_{\al}$, with projections $\pi_{\al}$.  Since  $K\in \cD_{\on{ess}-c}(X)$ we conclude from \rl{adm-essc} that there exists an index $\al$ and $K_{\al}\in\cD_{\on{ess}-c}(X_{\al})$ such that $K\simeq \pi_{\al}^!K_{\al}$. Moreover, since the projection $\pi_{\al}:X(k)\to X_{\al}(k)$ is surjective, our assumption implies that $\iota_{x_{\al}}^{!}(K_{\al})\in \mathstrut^{p}\cD^{\geq 0}(\pt)$ for all  $x_{\al}\in X_{\al}(k)$. Therefore, by the assertion for $X_{\al}$, shown above, we get $K_{\al}\in \mathstrut^{p}\cD^{\geq 0}(X_{\al})$. Since $\pi_{\al}^!$ is $t$-exact, the assertion follows.
\end{proof}

\begin{Emp} \label{E:pervfun}
{\bf Perversity function.}

\smallskip

(a) By a {\em perversity} on an $\cI$-stratified $\infty$-stack $(\cY,\{\cY_{\al}\}_{\al\in\cI})$ (see Section~\re{plstr}),
we mean a function $p_{\nu}:\cI\to \B{Z}:\al\mapsto\nu_{\al}$, or, what is the same, a collection $\{\nu_{\al}\}_{\al\in\cI}$ of integers.

\smallskip

(b) Let $f:\cX\to \cY$ be a morphism of $\infty$-stacks, where $\cX$ is a placid $\infty$-stack,
and $(\cY, \{\cY_{\al}\}_{\al\in\cI})$  is a placidly stratified $\infty$-stack, satisfying assumptions (i) and (ii) of Section~\re{small}(c).
Then $f$ gives rise to perversity $p_{f}:=\{\nu_{\al}\}_{\al\in\cI}$, defined by $\nu_{\al}:=b_{\al} +\dt_{\al}$ for all $\al\in\cI$.
\end{Emp}

\begin{Emp} \label{E:pervplstr}
{\bf Perverse $t$-structures on placidly stratified $\infty$-stacks.}

\smallskip

(a) Following \cite[Definition~5.5.1]{BKV},  we say that an $\infty$-stack $\cY$ {\em admits gluing of sheaves}, if for every topologically
fp-locally closed embedding $\eta:\cX\to\cY$, the pushforward $\eta_*:\cD(\cX)\to \cD(\cY)$ (see \cite[Section~5.4.4]{BKV}) admits a
left adjoint $\eta^* :\cD(\cY)\to \cD(\cX)$.

\smallskip

(b) Let $(\cY,\{\cY_{\al}\}_{\al\in\cI})$ be a placidly stratified $\infty$-stack such that $\cY$ admits gluing of sheaves. As every $\cY_{\al}$ is placid, every $\cD(\cY_{\al})$ has a canonical perverse  $t$-structure $(\mathstrut^{p}\cD^{\leq 0}(\cY_{\al}),\mathstrut^{p}\cD^{\geq 0}(\cY_{\al}))$ (see Section~\re{tstr}). Since $\cY$ admits gluing of sheaves, we have pullback functors $\eta_{\al}^{*},\eta_{\al}^{!}:\cD(\cY)\to\cD(\cY_{\al})$.

\smallskip

Then, by \cite[Proposition~6.4.2]{BKV}, for every perversity function  $p_{\nu}=\{\nu_{\al}\}_{\al\in\cI}$ there exists a unique $t$-structure $(\mathstrut^{p_{\nu}}\cD^{\leq 0}(\cY),\mathstrut^{p_{\nu}}\cD^{\geq 0}(\cY))$ on $\cD(\cY)$ such that
\begin{equation*}
\mathstrut^{p_{\nu}}\cD^{\leq 0}(\cY)=\{K\in\cD(\cY)~\vert~\eta_{\al}^{*}K\in\mathstrut^{p}\cD^{\leq -\nu_{\al}}(\cY_{\al})\text{ for all }\al\in\cI\},
\end{equation*}
\begin{equation*}
\mathstrut^{p_{\nu}}\cD^{\geq 0}(\cY)=\{K\in\cD(\cY)~\vert~\eta_{\al}^{!}K\in\mathstrut^{p}\cD^{\geq -\nu_{\al}}(\cY_{\al})\text{ for all }\al\in\cI\}.
\end{equation*}

(c) We say that a sheaf $K\in \cD(\cY)$ is {\em $p_{\nu}$-perverse}, if it is perverse with respect to perversity $p_{\nu}$.
\end{Emp}

\begin{Emp} \label{E:interm}
{\bf The intermediate extension.} Suppose we are in the situation of Section~\re{pervplstr}(b), and let $\cU\subseteq\cY$ be an
fp-open $\infty$-substack.

\smallskip

(a) Then $\cU\subseteq\cY$ is equipped with a constructible stratification $\{\cU_{\al}\}_{\al\in\cI_{\cU}}$, where   $\cU_{\al}:=\cU\cap \cY_{\al}$ for every $\al\in \cI$ and $\cI_{\cU}:=\{\al\in\cI~\vert~\cU_{\al}\neq\emptyset\}$. Moreover,
$(\cU,\{\cU_{\al}\}_{\al\in\cI_{\cU}})$ is a placidly stratified $\infty$-stack.

\smallskip

(b) The perversity function  $p_{\nu}:\cI\to\B{Z}$ restricts to a perversity function $p^{\cU}_{\nu}:\cI_{\cU}\to\B{Z}$, thus gives rise to a
perverse $t$-structure $(\mathstrut^{p^{\cU}_{\nu}}\cD^{\leq 0}(\cU),\mathstrut^{p^{\cU}_{\nu}}\cD^{\geq 0}(\cU))$ on $\cD(\cU)$.

\smallskip

(c) For every $p_{\nu}$-perverse sheaf $K\in \cD(\cY)$, its restriction $K|_{\cU}\in \cD(\cU)$ is $p^{\cU}_{\nu}$-perverse. Conversely, for every $p^{\cU}_{\nu}$-perverse sheaf $K_U\in \cD(\cU)$ there exists a unique extension $K\in\cD(\cY)$ (called the {\em intermediate extension}) such that $K$ is $p_{\nu}$-perverse and $K$ does not have non-zero subobjects and quotients supported on $\cY\sm\cU$ (see \cite[Corollary~6.4.9]{BKV}).
\end{Emp}

\begin{Emp} \label{E:!ls}
{\bf $!$-local systems.}

\smallskip

(a) For every $\infty$-prestack $\C{X}$, we define a full subcategory $\Loc^!(\cX)\subseteq\cD(\cX)$ of {\em $!$-local systems} defined as follows:
\smallskip

\quad $\bullet$ For $Y\in \Affft_k$, we denote by $\Loc^!(Y)$ be the full $\infty$-subcategory of $\cD(Y)$ consisting of objects
$\omega_{Y,\cL}:=\omega_{Y}\otimes \cL$, where $\cL$ is a usual (ind-constructible) local system on $Y$ and $\om_Y$ is a dualizing complex.
For every morphism $f:Y'\to Y$ of $\Affft_k$, we have a natural isomorphism
$f^!(\omega_{Y,\cL})\simeq \omega_{Y',f^*\cL}$, thus the class of $!$-local systems is closed under $!$-pullbacks.

\smallskip

\quad $\bullet$ Next, for $X\in \Aff_k$, we denote by $\Loc^!(X)$ be the full $\infty$-subcategory of $\cD(X)$ consisting of objects
$f^!K$, where $f:X\to Y$ is a morphism with $Y\in \Affft_k$ and $K\in\Loc^!(Y)$. Finally, for an arbitrary $\infty$-prestack $\cX$, we denote by
$\Loc^!(\cX)$ be the full $\infty$-subcategory of $\cD(\cX)$ consisting of objects $K$ such that $f^!K\in \Loc^!(X)$ for every morphism
$f:X\to \cX$ from an affine scheme $X$.

\smallskip

(b) By construction, the class of $!$-local systems is closed under $!$-pullbacks and $!$-tensor products. In particular,
for every morphism $f:\cX\to Y$ with $Y\in \Affft_k$ and every local system $\cL$ on $Y$, the pullback $\om_{\cX,\cL}:=f^!(\om_{Y,\cL})$ is a $!$-local system. To simplify the notation, we will often write $\om_{\cL}$ instead of $\om_{\cX,\cL}$.

\smallskip

(c) For every $Y\in\Algft_k$ and $K\in\Loc^!(X)$, 
we have $K\in \cD^{\geq -2\dim Y}(Y)$, where $\cD^{\geq -2\dim Y}$ refers to the usual (rather than perverse) $t$-structure.  More generally, mimicking the proof of \cite[Lemma~6.3.5(a)]{BKV}, for every placid $\infty$-stack $\cX$ and $K\in\Loc^!(\cX)$ we have
$K\in \mathstrut^{p}\cD^{\geq 0}(\cX)$.

\smallskip

(d) Similarly, mimicking the proof of \cite[Lemma~6.3.5(d)]{BKV}, one shows that if $\eta:\cX\hra\cY$ is an fp-locally closed immersion between placid $\infty$-stacks of codimension $d$ with $\cY$ smooth, then for every  $K\in\Loc^!(\cY)$ we have $\eta^{*}K\in \mathstrut^{p}D^{\leq -2d}(\cX) $.
\end{Emp}

The following result is a slight generalization of \cite[Theorem~6.5.3]{BKV}.

\begin{Thm}\label{g-small}
Let $(\cY,\{\cY_{\al}\}_{\al\in\cI})$ be a placidly stratified $\infty$-stack, let $\C{U}\subseteq \C{Y}$ be an fp-open $\{\cY_{\al}\}_{\al\in\cI}$-adapted $\infty$-substack, and let $f:\cX\to\cY$ be a $\C{U}$-small morphism of $\infty$-stacks such that
$f$ is ind-fp-proper, $\cX$ is smooth, while $\cY$ admits gluing of sheaves.

Then for every $K\in\Loc^{!}(\cX)$, the pushforward
$f_!(K)$ is $p_f$-perverse. Moreover, $f_!(K)$ is the intermediate extension of its restriction to $\C{U}$.
\end{Thm}
\begin{proof}
To proof the result, we repeat the argument of \cite[Theorem~6.5.3]{BKV}, except we replace \cite[Lemma~6.3.5(a),(d)]{BKV} by  Sections~\re{!ls}(c),(d).
\end{proof}

The following general lemma will be used later.

\begin{Lem}\label{r-tex}

(a) Let $\cC$ be a stable $\infty$-category equipped with $t$-structure. Then there exists a unique
$t$-structure on $\cC^{B\Gm}$  such that the forgetful functor $\on{For}:\cC^{B\Gm}\to \cC$ is $t$-exact.

\smallskip 

(b) Assume further that $\C{C}$ is $\qlbar$-tensored, cocomplete and that the $t$-structure on $\cC$ is compatible with filtered colimits.
Then for every $\tau\in\on{Rep}_{\bql}(\Gm)$, the functor $\on{coinv}_{\tau}:\C{C}^{B\Gm}\to\C{C}$ is right $t$-exact.
\end{Lem}                                                                                                                                         \begin{proof}                                                                                                                                    (a) Is a particular case of a standard fact, asserting that for every $\infty$-category $\cI$ the stable $\infty$-category of functors
$\Fun(\cI,\cC)$ has a unique $t$-structure such that $\Fun(\cI,\cC)^{\leq 0}=\Fun(\cI,\cC^{\leq 0})$ and
$\Fun(\cI,\cC)^{\geq 0}=\Fun(\cI,\cC^{\geq 0})$.

\smallskip

(b) Since the $t$-structure is compatible with filtered colimits, the functor $\tau\otimes_{\bql}-:\C{C}\to \C{C}$ is $t$-exact. As
$\on{coinv}_{\tau}=\on{coinv}_{\Gm}\circ(\tau\otimes_{\bql}-)$, it suffices to show that the functor of $\Gm$-coinvariants is right $t$-exact.
It is thus sufficient to prove that its right adjoint $\triv:\C{C}\to\C{C}^{B\Gm}$
is left $t$-exact. But it is $t$-exact, as the composition $\For\circ\triv$ is the identity functor.
\end{proof}

\subsection{Lemmas on quasi-coherent sheaves}\label{app}

Motivated by a nice trick explained to us by Zhiwei Yun, in this section we will establish simple properties of quasi-coherent sheaves, playing a central role in the proof of Proposition~\ref{p-tau2}.


\begin{Lem} \label{coh-loc}
Let $i:X\hra Y$ be a closed embedding of codimension $d$ between smooth connected schemes over $k$. Then

\smallskip

(a) the derived pullback functor $i^{*}: D^{+}(\on{QCoh}(Y))\to D^{+}(\on{QCoh}(X))$ sends $D^{\geq 0}$ to $D^{\geq -d}$;

\smallskip
		
(b) for every non-zero sheaf $K\in\on{QCoh}(Y)$, which is scheme-theoretically supported on $X$, we have $\cH^{-d}(i^{*}K)\neq 0$.
\end{Lem}
\begin{proof}
(a) follows from the Koszul resolution of $i_*\mathcal{O}_X$.

\smallskip

(b) By assumption, one has $K=i_{*}M$ for some $M\in\on{QCoh}(X)$. Thus, by the projection formula, we have $i^{*}i_{*}M\simeq M\otimes_{\mathcal{O}_{X}}i^{*}i_{*}\mathcal{O}_X$, and the assertion follows from the Koszul resolution for $i_*\mathcal{O}_X$ as well.
\end{proof}

\begin{Cor}\label{coh2}
In the situation of Lemma~\ref{coh-loc}, let $K\in D^{\geq 0}(\on{QCoh}(Y))$ be such that $\cH^{0}(K)$ has a non-zero subsheaf $K'$, which is scheme-theoretically supported on $X$. Then $\cH^{-d}(i^{*}K)\neq 0$.
\end{Cor}
\begin{proof}
Consider the fiber sequence
\[i^{*}\cH^{0}(K)\to i^{*}K\to i^{*}\tau^{> 0}K.\]
Since $i^{*}\tau^{> 0}K$ belongs to $D^{\geq -d+1}$ by Lemma~\ref{coh-loc}(a), we conclude that $\cH^{-d}(i^{*}\cH^0(K))\simeq\cH^{-d}(i^{*}K)$. Thus, replacing $K$ by $\cH^{0}(K)$, we can assume that $K\in\on{QCoh}(Y)$. Then, using Lemma \ref{coh-loc}(a) for $i^{*}(K/K')$, we obtain that the map
\[\cH^{-d}(i^{*}K')\to \cH^{-d}(i^{*}K)\]
is injective. Since  $\cH^{-d}(i^{*}K')\neq 0$ by Lemma~\ref{coh-loc}(b), the assertion follows.
\end{proof}

\begin{Prop}\label{coh3}
Let $\eta_1,\dots, \eta_m:X\to Y$ be closed embeddings of smooth connected $k$-schemes, let $p_X:X\times Y\to X$ and
 $p_Y:X\times Y\to Y$ be projections, and let $K\in D^{+}(\on{QCoh}(X\times Y))$ be such that

\smallskip

(i)	$K$ is set-theoretically supported on the union of the graphs of the $\eta_a$'s;

\smallskip
	
(ii)	we have $K\otimes p_{X}^{*}A\in D^{\geq 0}$ for all $A\in\on{QCoh}(X)$.

\smallskip

\noindent Then we have $K \otimes p_{Y}^{*}B\in D^{\geq \dim X-\dim Y}$ for all $B\in\on{QCoh}(Y)$.
\end{Prop}

\begin{proof}
Let $n$ be the minimal integer $j$ such that $\cH^{j}(K \otimes p_{Y}^{*}B)\neq 0$. We want to show that $n\geq \dim X-\dim Y$. By assumption~(i), there exists a nonzero subsheaf $K'\subseteq \cH^{n}(K\otimes p_{Y}^{*}B)$, which is scheme-theoretically supported on the graph of some $\eta_a$.

Consider morphism $\Gm_{\eta_a}:=(\Id_X,\eta_a):X\to X\times Y$.  It is a closed embedding of smooth connected $k$-schemes of codimension $\dim Y$.
Thus, by Corollary \ref{coh2}, we have $\cH^{n-\dim Y}(\Gm_{\eta_a}^{*}(K \otimes p_{Y}^{*}B))\neq 0$. So it suffices to show that $\Gm_{\eta_a}^{*}(K\otimes p_{Y}^{*}B)\in D^{\geq \dim X-2\dim Y}$.

Using isomorphisms
\[\Gm_{\eta_a}^{*}(K \otimes p_{Y}^{*}B)\simeq \Gm_{\eta_a}^{*}(K) \otimes \eta_{a}^{*}B\simeq \Gm_{\eta_a}^{*}(K \otimes p_{X}^{*}(\eta_{a}^{*}B))\]
and Lemma~\ref{coh-loc}(a), it suffices to prove that $K \otimes p_{X}^{*}(\eta_a^{*}B)\in D^{\geq \dim X-\dim Y}$.
By Lemma~\ref{coh-loc}(a), we have $\eta_{a}^{*}B\in D^{\geq \dim X-\dim Y}$, so the assertion follows from assumption (ii).
\end{proof}

\section{Group analogs of results of \cite{BKV}}

In this section, we will show analogs of the results \cite{BKV} in the group case.

\subsection{Flatness of the Chevalley map for arcs}

\begin{Emp} \label{E:looparc}
{\bf Arc and loop spaces.} We set $\co:=k[[t]]$ and $F:=k((t)))$.

\smallskip

(a) For an affine scheme of finite type $X$ over $\co$, we denote by $\clp(X)$ the affine scheme over $k$ (called  the {\em arc space} of $X$),
representing a functor $A\mapsto X(A[[t]])$. Then $\clp(X)\simeq\lim_{n\in\B{N}}\cL_n(X)$, where
$\cL_n(X)$ is an affine scheme of finite type over $k$ (called the {\em $n$-truncated arc space} of $X$), representing a functor $A\mapsto X(A[t]/(t^{n+1}))$.

\smallskip

(b) We have a closed embedding $X\hra \cL_n (X)$, corresponding to the $A$-algebra homomorphism
$A\hra A[t]/(t^{n+1})$, and projections $\ev_n:\cL_n (X)\to X$ and $\ev:\clp(X)\to X$, corresponding to the $A$-algebra homorphisms $A[t]/(t^{n+1})\to A$ and $A[[t]]\to A$ such that $t\mapsto 0$.

\smallskip

(c) We have a natural action of multiplicative group $\B{G}_m$ on $\cL_n (X)$ over $X$, which corresponds to the action of $A^{\times}$ on the $A$-algebra $A[t]/(t^{n+1})$ given by the rule $a\cdot f(t):=f(at)$. Then for $x\in \cL_n (X)(k)$ and $a\in \B{G}_m$,  the limit $\lim_{a\to 0}(a\cdot x)$ exists and is equal to $\ev_n(x)$.

\smallskip

(d) If $X$ is smooth over $\cO$, then the argument of \cite[Section~3.1.4]{BKV} shows that the arc space $\cL(X)$ is admissible and  $\cL(X)\simeq\lim_n\cL_n(X)$ is an admissible presentation.

\smallskip

(e) For every affine scheme of finite type $X$ over $F$, we denote by $\cL(X)$ the ind-scheme over $k$ representing functor
$A\mapsto X(A((t)))$, which is called the {\em loop  space} of $X$ (compare \cite[Section~3.1.5]{BKV}).
\end{Emp}

\begin{Emp} \label{E:setup}
{\bf Set up.}

\smallskip

(a) Let $G$ be a connected reductive group over $k$, let $B\supseteq T$ be a Borel group and a maximal torus of $G$, respectively, let $W$\label{N:W} be the Weyl group of $G$, let $\La:=X_{*}(T)$\label{N:X*T} be the lattice of cocharacters of $T$, let $\widetilde{W}:=\La\rtimes W$\label{N:Waff} be the extended affine Weyl group of $G$, and let $R$\label{N:R} be the set of roots of $G$.

\smallskip

(b) By a theorem of Steinberg--Chevalley \cite{S}, the restriction map $k[G]\to k[T]$ induces an isomorphism $k[G]^{G}\isom k[T]^{W}$, where $G$ acts by conjugacy. Let $\kc:=T/W=\Spec(k[G]^G)$\label{N:c} be the Chevalley space of $G$. Then we have canonical projections  $\chi:G\rightarrow \kc$\label{N:chi} and $\pi:T\rightarrow\kc$.

\smallskip

(c) Let $\mathfrak{D}:=\prod_{\alpha\in R}(1-\alpha)$\label{N:mathfrakD} be the discriminant function.
Then $\mathfrak{D}\in k[\kc]=k[T]^W$, and the regular semisimple locus $\kc^{\rs}\subseteq\kc$\label{N:crs} is the complement of the locus
$Z(\mathfrak{D})$ of zeros of $\mathfrak{D}$. We denote by  $G^{\rs}:=\chi^{-1}(\kc^{\rs})$\label{N:grs} and $T^{\rs}:=\pi^{-1}(\kc^{\rs})$\label{N:trs} the preimages of $\kc^{\rs}$.

\smallskip

(d) For most of the paper (everywhere except \rl{eqdim}) we will assume that the derived group $G^{\der}$ of $G$ is simply connected. In this case, $\kc$ is a smooth affine scheme (by \cite[Theorem~6.1]{S}). Furthermore, this assumption implies that the centralizer of every semisimple
$g\in G$ is connected (see \cite[$\S$~2.11, Theorem]{H}) that centralizer of every $g\in G^{\rs}$ is a maximal torus and that the restriction $\pi^{\rs}:T^{\rs}\to \fc^{\rs}$ of $\pi$ is a $W$-torsor (see \cite[$\S$~2.5, Remark]{H}).

\smallskip

(e) We denote by $I:=\ev^{-1}(B)\subseteq\clp (G)$\label{N:I} the Iwahori subgroup scheme of $\cL G$ and denote by
$v:I\to\clp(\kc)$ the restriction of $\chi:\kg\rightarrow \kc$. Similarly, for every $n\in\B{N}$, we set $I_n:=\ev_n^{-1}(B)\subseteq \cL_n(G)$ and let $v_n:I_n\to\cL_n(\kc)$ be the restriction of $\chi_n:= \cL_{n}(\chi):\cL_{n}(G)\to\cL_n(\kc)$.

\smallskip

(f) Following \cite{Yun1, Yun}, we assume that the characteristic of $k$ is either zero  or greater than $2h$, where $h$ is the Coxeter number of $G$. Note that this assumption implies that the characteristic of $k$ is prime to $|W|$, hence prime to the cardinality of $Z(G^{\der})$.
\end{Emp}

The following result is central for what follows.

\begin{Thm}\label{flat1}
For every $n\in\NN$, the morphisms $\chi_n:\cL_{n}(G)\to\cL_n(\kc)$ and $v_n:I_n\to\cL_n(\kc)$ are flat.
\end{Thm}

Our strategy will be to deduce the result from the corresponding results for Lie algebras shown in \cite[Theorem~3.4.7]{BKV}.\footnote{Notice that unlike many of the results of \cite{BKV}, which are proven under an assumption that the characteristic of $k$ is prime to $|W|$, our proof of \cite[Theorem~3.4.7]{BKV} uses \cite[Corollary~2.5.2]{Yun1} and therefore is only valid under the stronger assumption of Section~\re{setup}(f).}

\begin{Lem} \label{L:eqdim}
Assume that $G$ satisfies the assumption of Section~\re{setup}(f) but not necessary of Section~\re{setup}(d), and
let $\C{U}_G\subseteq G$ be the locus of unipotent elements of $G$. Then we have equalities
\begin{equation} \label{Eq:eqdim}
\dim\cL_n(\C{U}_G)=\dim\cL_{n}(G)-\dim \cL_n(\kc),\quad \dim(\cL_n(\C{U}_G)\cap I_n)=\dim I_n-\dim \cL_n(\kc).
\end{equation}
\end{Lem}
\begin{proof}
Let $G^{\on{sc}}$ be the semisimple covering of the derived group of $G$, and let $\C{N}_{\fg}$ be the variety of nilpotent elements of
$\fg:=\Lie G$.

Since the projection $G^{\on{sc}}\to G$ induces an isomorphism  $\C{U}_{G^{\on{sc}}}\isom\C{U}_G$, we can replace $G$ by $G^{\on{sc}}$,
thus assuming that $G$ is semisimple and simply connected. Then, applying results of Springer \cite{Spr} (or
using quasi-logarithms of \cite[Section~1.8]{KV} and results of \cite{BR}) one shows that there exists an $\Ad G$-equivariant isomorphism
$\Phi:\C{U}_{G}\isom\C{N}_{\fg}$, inducing an isomorphism $R_u(B)\isom \Lie R_u(B)$. Therefore $\Phi$ induces isomorphisms
\begin{equation} \label{Eq:isom2}
\cL_n(\C{U}_G)\isom \cL_n(\C{N}_{\fg}), \quad \cL_n(\C{U}_G)\cap I_n\isom\cL_n(\C{N}_{\fg})\cap\Lie I_n.
\end{equation}

Since $\C{N}_{\fg}=\chi_{\fg}^{-1}(0_{\fc})$, where $\chi_{\fg}:\fg\to\fc_{\fg}$ is the Chevalley map for $\fg$ and
$0_{\fc}:=\chi_{\fg}(0)$, and since  $\cL_n$ commutes with limits, we have $\cL_n(\C{N}_{\fg})=\cL_{n}(\chi_{\fg})^{-1}(0_{\fc})$.
Therefore equalities \form{eqdim} follows from isomorphisms \form{isom2} and flatness of morphisms
$\cL_n(\chi_{\fg}):\cL_n(\fg)\to\cL_n(\fc_{\fg})$ and $v_n:\Lie I_n\to\cL_n(\fc_{\fg})$, established in \cite[Theorem~3.4.7]{BKV}.
\end{proof}

\begin{Cor} \label{C:eqdim}
In the situation of Theorem~\ref{flat1}, for every $a\in\fc(k)$, we have equalities
\begin{equation} \label{Eq:eqdim2}
\dim\cL_n(\chi^{-1}(a))=\dim\cL_{n}(G)-\dim \cL_n(\kc), \quad \dim(\cL_n(\chi^{-1}(a))\cap I_n)=
\dim I_n-\dim \cL_n(\kc).
\end{equation}
\end{Cor}
\begin{proof}

To show the first equality \form{eqdim2}, we choose any $s\in T(k)\cap \chi^{-1}(a)$. Then $\chi^{-1}(a)\subseteq G$ consists of elements $g$ with Jordan decomposition $g=g_s u_s$
such that $g_s$ is $G$-conjugate to $s$. Since $G^{\der}$ is assumed to be simply connected, the centralizer $G_s$ is connected. Moreover, the map $[(g,u)]\mapsto g(su)g^{-1}$ induces an isomorphism
\begin{equation} \label{Eq:is}
G\times^{G_{s}} \cU_{G_s}\isom \chi^{-1}(a),
\end{equation}
(see, for example, \cite[Ch. II, 3.10, Lemma and Thm]{Slo}).

Therefore it induces an isomorphism
\begin{equation} \label{Eq:isom4}
\cL_{n}(G)\times^{\cL_{n}(G_s)}\cL_n(\cU_{G_s})\isom\cL_n(G\times^{G_{s}} \cU_{G_s})\isom\cL_n(\chi^{-1}(a)),
\end{equation}
where for the left isomorphism we use \cite[Section~3.1.4(f)]{BKV}. Then
\[
\dim\cL_n(\chi^{-1}(a))=\dim \cL_n(G)-\dim \cL_n(G_s)+\dim\cL_n(\cU_{G_s}),
\]
so the first equality \form{eqdim2} is equivalent to the first equality \form{eqdim} for $G_s$. Since our assumption
of Section~\re{setup}(f) for $G$ implies that $G_s$, the assertion thus follows from \rl{eqdim}.

\smallskip

To show the second equality \form{eqdim2}, it suffices to show that the maps $[(g,u)]\mapsto g(su)g^{-1}$ for $s\in T(k)\cap\chi^{-1}(a)$ induce an isomorphism
\begin{equation} \label{Eq:isom3}
\coprod_{s\in T(k)\cap\chi^{-1}(a)}\left(I_n\times^{\cL_{n}(G_{s})\cap I_n}(\cL_n(\C{U}_{G_{s}})\cap I_n)\right)\isom \cL_n(\chi^{-1}(a))\cap I_n.
\end{equation}
Indeed, assuming isomorphism \form{isom3}, the second equality \form{eqdim2} is equivalent to the second equality \form{eqdim} for $G_s$
and thus would follow from \rl{eqdim}.

\smallskip

Notice first that the maps $[(g,u)]\mapsto g(su)g^{-1}$ induce an isomorphism
\begin{equation} \label{Eq:isom5}
\coprod_{s\in T(k)\cap \chi^{-1}(a)}\left(B\times^{G_{s}\cap B}R_u(G_{s}\cap B)\right)\isom \chi^{-1}(a)\cap B.
\end{equation}
Indeed, this follows from the fact that for every $g\in\chi^{-1}(a)\cap B$ its semisimple part $g_s\in B$ is $B$-conjugate to
a unique element $s\in T\cap \chi^{-1}(a)$.

Next, we observe that the diagram
\begin{equation} \label{Eq:cd}
\begin{CD}
I_n\times^{\cL_{n}(G_{s})\cap I_n}(\cL_{n}(\C{U}_{G_{s}})\cap I_n) @>>> \cL_n(G)\times^{\cL_n(G_s)}\cL_{n}(\C{U}_{G_{s}})\\
@VVV @VVV\\
B\times^{G_{s}\cap B}R_u(G_{s}\cap B) @>>> G\times^{G_{s}}\C{U}_{G_{s}},
\end{CD}
\end{equation}
where the horizontal maps are induced by the embedding $B\hra G$ and vertical maps by the projection $\ev_n:\cL_n(G)\to G$, is Cartesian.

Finally, isomorphism \form{isom3} follows from a combination of isomorphisms \form{is}, \form{isom4} and \form{isom5}, using the Cartesian diagram \form{cd}.
\end{proof}

Now we are ready to prove Theorem~\ref{flat1}.

\begin{proof}[Proof of Theorem~\ref{flat1}]
First we will show the case of $\chi_n$. As the source and the target of $\chi_n$ are smooth and connected, it suffices to prove that all non-empty geometric fibers of $\chi_n$ are equidimensional of constant dimension $\dim\cL_{n}(G)-\dim \cL_n(\kc)$.

Since the function
\begin{equation} \label{Eq:dim}
\cL_n(G)\to\B{Z}:g\mapsto \dim_g \chi_n^{-1}(\chi_n(g))
\end{equation}
is upper semi-continuous (see \cite[Theorem 13.1.3]{EGAIV} or \cite[Tag~02FZ]{Sta}), it thus suffices to show
that
\begin{equation} \label{Eq:ineq}
\dim_g \chi_n^{-1}(\chi_n(g))\leq\dim\cL_{n}(G)-\dim \cL_n(\kc)
\end{equation}
for every $g\in \cL_n(G)(k)$.

Recall that group $\B{G}_m$ acts on $\cL_n(G)$ and $\cL_n(\fc)$ (see Section~\re{looparc}(c)), the morphism
$\chi_n$ is $\B{G}_m$-equivariant, and element $\ov{g}:=\ev_n(g)\in G(k)$ lies in the closure of the $\B{G}_m$-orbit of $g\in\cL_n(G)(k)$.

Thus, by the semi-continuity of function \form{dim} we conclude that
\[
\dim_g \chi_n^{-1}(\chi_n(g))\leq \dim_{\ov{g}} \chi_n^{-1}(\chi_n(\ov{g})).
\]
Therefore it is suffices to show inequality \form{ineq} for $g\in G(k)$.

Finally, since functor $\cL_n$ commutes with limits, we have an identification
\[
\chi_n^{-1}(\chi_n(g))\simeq\chi_n^{-1}(a)\simeq\cL_n(\chi^{-1}(a))\text{ for }a:=\chi(g)\in \fc(k).
\]
Thus, inequality \form{ineq} for $g\in G(k)$ follows from the first equality \form{eqdim2}.

\smallskip

The proof for $v_n$ is similar. First of all, it suffices to show that for every $g\in I_n$
we have inequality
\begin{equation} \label{Eq:ineq2}
\dim_g v_n^{-1}(v_n(g))\leq\dim I_n-\dim \cL_n(\kc).
\end{equation}
Next, since $I_n\subseteq \cL_n(G)$ is closed and $\B{G}_m$-invariant, repeating the argument for $\chi_n$ we conclude that it suffices 
to show inequality \form{ineq2} for $g\in B(k)$. Finally, using identification
\[
v_n^{-1}(v_n(g))\simeq \cL_n(\chi^{-1}(a))\cap I_n\text{ for }a:=\chi(g)\in \fc(k),
\]
inequality \form{ineq2} follows from the second equality \form{eqdim2}.
\end{proof}


\subsection{The Goresky--Kottwitz--MacPherson stratification}

\begin{Emp} \label{E:constr}
{\bf Construction} (compare \cite[Sections~3.3.2--3.3.3]{BKV}).

\smallskip

(a) We set $a:=|W|$, $\cO':=k[[t^{\frac{1}{a}}]]$, let $\xi\in k$ be a fixed $a$-th root of unity, and let $\sigma\in\Aut(\cO'/\cO)$ the automorphism given by $t^{\frac{1}{a}}\mapsto\xi t^{\frac{1}{a}}$.  We denote by $T':=\Res_{\cO'/\cO}T$ the Weil restriction of scalars, and
for $w\in W$ we consider the subscheme of fixed points $T_{w}:=\Res_{\cO'/\cO} (T)^{w\sigma}$. Since the characteristic of $k$ is assumed to be prime to $a$ (see Section~\re{setup}(f)), each $T_w$ is a smooth group scheme over $\cO$.

\smallskip

(b) To every pair $(w,\br)$, where $w\in W$ and $\br$ is a function $R\to\frac{1}{a}\B{Z}_{\geq 0}$, one associates the fp-locally closed subscheme $T_{w,\br}\subseteq\clp(T_{w})$ such that
\[
T_{w,\br}=\{t\in\clp(T_{w})~\vert~\val(1-\al(t))=r(\al)\text { for all }\al\in R \}.
\]
Thus, $T_{w,\br}$ is either empty or an open subscheme in some congruence subgroup scheme of $\clp(T_{w})$, in which case $T_{w,\br}$ is connected and strongly pro-smooth.
Moreover, by Sections~\re{admstack}(e) and \re{looparc}(c), scheme $T_{w,\br}$ is admissible. We will call $(w,\br)$ a {\em GKM pair}, if
$T_{w,\br}\neq\emptyset$.

\smallskip

(c) The projection $\pi:T\to\kc$ of affine schemes over $k$ induces a projection $\pi_{w}:T_w\to\kc$ of affine schemes over $\cO$, hence a projection $\pi_{w,\br}:T_{w,\br}\hra\clp(T_w)\to\clp(\fc)$.

\smallskip

(d) In the case $(w,\br)=(1,0)$, we have an equality $T_{w,\br}=\clp (T^{\rs})$.

\smallskip

(e) For every $u\in W$ and a GKM pair $(w,\br)$, the automorphism $u_*$ of $T$ induces an isomorphism $T_{w,\br}\isom  T_{uwu^{-1},u(\br)}$ over $\clp(\fc)$. In particular, we have an action of $W$ on the set of GKM pairs given by the formula $u(w,\br):=(uwu^{-1},u(\br))$, and we denote by $W_{w,\br}\subseteq W$ the stabilizer of $(w,\br)$.
\end{Emp}

\begin{Prop}\label{defwr}

For every GKM pair $(w,\br)$, the schematic image $\kc_{w,\br}$ of $\pi_{w,\br}$ is a connected affine strongly pro-smooth fp-locally closed subscheme of $\clp(\kc)$, and the map $\pi_{w,\br}: T_{w,\br}\to\kc_{w,\br}$ is a $W_{w,\br}$-torsor.
\end{Prop}
\begin{proof}
Since $\pi:T\to\fc$ a finite morphism which is a $W$-torsor over $\kc^{\rs}$ (see Section~\re{setup}(d)), the assertion follows from
\cite[Lemma~3.2.6]{BKV} by the same arguments as in \cite[Section~3.3.4]{BKV}.
\end{proof}

\begin{Emp} \label{E:notcodim}
{\bf Notation.}
As in \cite[Section~3.4.1]{BKV}, for every GKM pair $(w,\br)$, we consider codimensions $a_{w,\br}:=\codim_{\clp (T_{w})}(T_{w,\br})$ and $b_{w,\br}:=\codim_{\clp(\kc)}(\kc_{w,\br})$, denote by $r$ the rank of $G$, and set $d_{\br}:=\sum_{\al\in R} \br(\al)$, $c_{w}:=r-\dim(T^{w})$ and  $\delta_{w,\br}:=\frac{d_{\br}-c_{w}}{2}$.
\end{Emp}

The following result is a group analog of \cite[Corollary~3.4.9(a)]{BKV}:

\begin{Lem}\label{L:codim}
For every GKM pair $(w,\br)$, the fp-locally closed subscheme $I_{w,\br}:=v^{-1}(\kc_{w,\br})\subseteq I$ is of pure codimension $b_{w,\br}$.

\end{Lem}
\begin{proof}
Repeating the argument of \cite[Corollary~3.4.9(a)]{BKV}, the assertion is a formal corollary of the flatness of $v_n$ (Theorem~\ref{flat1}).
\end{proof}

One has the following analog of \cite[Theorem~8.2.2]{GKM}:

\begin{Prop}\label{codim}
For every GKM pair $(w,\br)$, we have an equality $b_{w,\br}=\delta_{w,\br}+a_{w,\br}+c_w$.
\end{Prop}

The proof is based on the following simple result.

\begin{Lem} \label{L:tangent}
(a) Let $X$ be a scheme with placid presentation $X\simeq \lim_{\al}X_{\al}$, and let $\pr_{\al}:X\to X_{\al}$ be projections.
Then for every $x\in X(k)$, the tangent space $T_x(X)$ has a placid (even admissible) presentation $T_x(X)\simeq \lim_{\al}T_{x_{\al}}(X_{\al})$
with $x_{\al}:=\pr_{\al}(x)\in X_{\al}(k)$ for each $\al$.

\smallskip

(b) Let $f:X\to Y$ be an fp-morphism between strongly pro-smooth $k$-schemes. Then for every $x\in X(K)$, the differential $df_x:T_x(X)\to T_y(Y)$ is finitely presented, and we have an equality
\[
\un{\dim}_f(x)=\un{\dim}_{df_x}(0).
\]

\smallskip

(c) Let $f:X\to Y$ be a morphism of smooth schemes over $\C{O}$, let
$x\in X(\C{O})=L^+(X)(k)$ be such that the generic fiber $f_F:X_F\to Y_F$ is \'etale at $x$, let $df_x:T_x(X)\to T_{f(x)}(Y)$ be the differential (an $\C{O}$-linear map of relative tangent spaces), and let $\clp(f):\clp(X)\to\clp(Y)$ be the map between arc spaces.

\smallskip

Then $d(\clp(f))_x:T_x(\clp(X))\to T_{f(x)}(\clp(Y))$ is an fp-closed embedding of pure codimension $\val(\det df_x)$, where $\val(\det df_x)$ denotes the valuation of the determinant of $df_x$.
\end{Lem}

\begin{proof}
(a) The isomorphism  $T_x(X)\simeq \lim_{\al}T_{x_{\al}}(X_{\al})$ follows from the fact that functor of tangent spaces commutes with limits. To finish, we observe that each projection $X_{\beta}\to X_{\al}$ is smooth, hence the differential
$T_{x_{\beta}}(X_{\beta})\to T_{x_{\al}}(X_{\al})$ is surjective, thus smooth.

\smallskip

(b) By assumption, there is a Cartesian diagram
\[
\begin{CD}
X@>f>> Y\\
@Vp_XVV @VVp_YV\\
X' @>f'>> Y'
\end{CD}
\]
such that $X',Y'\in\Schft_k$ are smooth and vertical morphisms are strongly pro-smooth. Since functor of tangent spaces preserves fiber products,
the assertion for $f$ follows from the corresponding assertion for $f'$. Finally, since $X'$ and $Y'$ are smooth over $k$, the assertion for $f'$ follows. 

\smallskip

(c) Assume first that $X=Y=\B{A}^n_{\C{O}}$ and $f$ is an $\C{O}$-linear map. In this case, $f$ is injective and the assertion is easy.
Next, to deduce the general case follows from
the linear case, we observe that $d(\clp(f))_x$ is naturally identified with $\clp(df_x):\clp(T_x(X))\to \clp(T_{f(x)}(Y))$.
\end{proof}

Now we are ready to prove Proposition~\ref{codim}.

\begin{proof}[Proof of Proposition~\ref{codim}]
Let $x\in T_{w,\br}$ and $y:=\pi_{w,\br}(x)\in\fc_{w,\br}$. Then we have a commutative diagram
\[
\xymatrix{T_{x}(T_{w,\br})\ar[r]\ar[d]_{d(\pi_{w,\br})_x}&T_{x}(\clp (T_w))\ar[d]^{d(\pi_{w})_x}\\T_{y}(\kc_{w,r})\ar[r]&T_{y}(\clp(\kc)).}
\]

Since the map $\pi_{w,\br}:T_{w,\br}\to\kc_{w,\br}$ is \'etale (by Proposition~\ref{defwr}), the map $d(\pi_{w,\br})_x$ is an isomorphism (compare the proof of \rl{tangent}(b)). Since the generic fiber $(\pi_w)_F:(T_w)_F\to \fc_F$ is \'etale at $x\in T_w(\cO)$, the map
$d(\pi_w)_x$ is injective of codimension $\val(\det d(\pi_w)_x)$ (by \rl{tangent}(c)). Therefore, by \rl{tangent}(b), we have equalities
\[
b_{w,\br}=\codim_{T_{y}(\clp(\kc))}(T_{y}(\kc_{w,r}))=\codim_{T_{y}(\clp(\kc))}(T_{x}(T_{w,r}))=
\]
\[
=\codim_{T_{x}(\clp(T_w))}(T_{x}(T_{w,r}))+\codim_{T_{y}(\clp(\kc))}(T_{x}(\clp(T_w)))=a_{w,\br}+\val(\det d(\pi_w)_x).
\]
It thus suffices to show that
\[
\val(\det d(\pi_w)_x)=\delta_{w,\br}+c_w=\frac{d_{\br}+c_w}{2}.
\]
But this follows by the argument of \cite[Lemma~8.2.1]{GKM} except we have to replace the identity \cite[(2.3.1)]{GKM} by its group version \cite[Section~4.23]{H}.
\end{proof}

\begin{Emp}
{\bf Remark.}
Note that our proof of the codimension formula is much shorter than the one of \cite{GKM}, as we already know that the map $\pi_{w,\br}:T_{w,\br}\to\kc_{w,\br}$ is \'etale.
\end{Emp}

\begin{Cor}\label{cor-codim}
For every GKM pair $(w,\br)$, we have an inequality $b_{w,\br}\geq\delta_{w,\br}$. Moreover, equality holds if and only if $w=1$ and $\br=0$.
\end{Cor}
\begin{proof}
Since $a_{w,\br}, c_w\geq 0$, the inequality $b_{w,\br}\geq\dt_{w,\br}$ follows from Proposition~\ref{codim}.
Moreover, equality holds if and only if $c_w = a_{w,\br} = 0$. Furthermore, this happens if and only if
$w = 1$ and $T_{w,\br}=T_{\br}\subseteq \clp(T)$ is open. Finally, $T_{\br}\subseteq \clp(T)$ is open if and only if $\br=0$.
\end{proof}

\subsection{The affine Grothendieck-Springer fibration}
\begin{Emp} \label{E:fibr}
{\bf The fibration} (compare \cite[Sections~4.1.1--4.1.3]{BKV}).

\smallskip

(a) Let $\Fl:=\cL G/ I$ be the affine flag variety, and set
\[
\wt{\fC}:=\{([g],\g)\in\Fl\times\cL G~\vert~g^{-1}\g g\in I\}.
\]
By functoriality, the Chevalley morphism $\chi:G\to\kc$ induces a map $\cL\chi:\cL G\to\cL\kc$,
and we set $\kC:=(\cL\chi)^{-1}(\clp(\kc))\subseteq \cL G$. The second projection $\Fl\times\cL G\to\cL G$ induces a projection
$\frak{p}:\wt{\fC}\to\kC$.

\smallskip

(b) Note that $\wt{\fC}$ is the inverse image of the fp-closed substack $[I/I]\subseteq[\cL G/I]$ by the map
\[\Ad:\Fl\times \cL G\to[\cL G/I]:([g],\g)\mapsto [g^{-1}\g g],\]
where $I$ acts on the right hand side by conjugation. Therefore  $\wt{\fC}$ is an fp-closed ind-subscheme of $\Fl\times\cL G$.
Moreover, since $\Fl$ is an ind-fp-proper ind-scheme over $k$, the map $\frak{p}$ is ind-fp-proper (compare \cite[Lemma~4.1.4]{BKV}).

\smallskip

(c) The group ind-scheme $\cL G$ acts on  $\wt{\fC}$ by the rule
\[h([g],\g)=([hg],h\g h^{-1})\] and the map $\Ad$ from part~(b) induces an isomorphism of $\infty$-stacks $[\wt{\fC}/\cL G]\isom [I/I]$. Moreover, the map $\frak{p}$ is $\cL G$-equivariant and hence induces a map between quotient stacks
\begin{equation}
\ov{\frak{p}}:[\wt{\fC}/\cL G]\to[\kC/\cL G],
\label{equiv-prop}
\end{equation}
which is also ind-fp-proper (by part~(b) and \cite[Section~1.2.9(b)]{BKV}).

\smallskip

(d) The projection $\pr:I\to T$ induces a morphism $[\wt{\fC}/\cL G]\isom [I/I]\to T$, which we also denote by $\pr$. Explicitly,
$\pr$ send a class $[([g],\g)]$ to the class of $g^{-1}\g g\in I$.
\end{Emp}

\begin{Emp} \label{E:open}
{\bf The constructible stratification} (compare \cite[Section~4.1.7]{BKV}).

\smallskip

(a) Let $Z(\kD)\subseteq\fc$ be the locus of zeros of $\kD\in k[\fc]$, and set $\clp(\kc)_{\bullet}:=\clp(\kc)\sm\clp(Z(\kD))$. It is a non-quasi-compact open subscheme of $\clp(\kc)$, which has an open increasing covering \[\clp(\kc)_{\bullet}=\bigcup_{m\in\NN}\kc_{\leq m},\]
where $\kc_{\leq m}:=\{a\in\clp(\kc)~\vert~\val(\kD(a))\leq m\}$.

\smallskip

(b) We set $\kC_{\bullet}:=(\cL\chi)^{-1}(\kc_{\bullet})\subseteq\kC$, $\kC_{\leq m}:=(\cL\chi)^{-1}(\kc_{\leq m})\subseteq \kC_{\bullet}$, and $\wt{\fC}_{\bullet}=\frak{p}^{-1}(\kC_{\bullet})$. We can then write $\wt{\fC}_{\bullet}$ as an increasing union of open ind-schemes:
\begin{equation}
\kC_{\bullet}=\bigcup_{m\in\NN}\kC_{\leq m}.
\label{colimC}
\end{equation}
Then each $\kC_{\leq m}$ is an ind-placid ind-scheme, being an fp-locally closed sub-indscheme of an ind-placid ind-scheme $\cL G$.

\smallskip

(c) For every GKM pair $(w,\br)$, we set $\kC_{w,\br}:=\chi^{-1}(\kc_{w,\br})\subseteq\kC_{\bullet}$ and $\wt{\fC}_{w,\br}:=\frak{p}^{-1}(\kC_{w,\br})\subseteq\wt{\fC}_{\bullet}$. As in the Lie algebra case,
$\{\kc_{w,\br}\}_{w,\br}$ forms a bounded constructible stratification of $\cL(\kc)_{\bullet}$
and induces a bounded constructible stratification $\{[\fC_{w,\br}/\cL G]_{\red}\}_{w,\br}$ of $[\fC_{\bullet}/\cL G]$.

\smallskip

(d) The projection $\ov{\frak{p}}$ from Section~\re{fibr}(c) induces projections $\ov{\frak{p}}_{\bullet}:[\wt{\fC}_{\bullet}/\cL G]\to [{\fC}_{\bullet}/\cL G]$ and
$\ov{\frak{p}}_{w,\br}:[\wt{\fC}_{w,\br}/\cL G]_{\red}\to [{\fC}_{w,\br}/\cL G]_{\red}$.
\end{Emp}

\begin{Emp} \label{E:affsprfib}
{\bf Affine Springer fibers.}

\smallskip

(a) For every $\g\in\kC_{\bullet}(k)\subseteq G^{\rs}(F)$, the closed ind-subscheme
\[
\Fl_{\g}=\{[g]\in\Fl~\vert~g^{-1}\g g\in I\},
\]
is the preimage of $\frak{p}^{-1}(\g)$ and is usually called the {\em affine Springer fiber} of $\g$. 
By Section~\re{fibr}(b), $\Fl_{\g}$ is an ind-fp-proper ind-scheme over $k$.

\smallskip

(b) As in the Lie algebra case, the reduced ind-scheme $(\Fl_{\g})_{\red}$ is a finite-dimensional scheme, locally of finite type over $k$. Moreover, if $\g\in\kC_{w,\br}(k)$, then, in the notation of
Section~\re{notcodim}, we have an equality $\dim(\Fl_{\g})_{\red}=\delta_{w,\br}$.


\smallskip

Namely, both assertions can be either deduced from the results of \cite{KL} and \cite{Be} by the argument of \cite[Appendix~B.2--B.3]{BV}
(using quasi-logarithms and the topological Jordan decomposition) or be obtained in a more general framework of \cite{Bt1} and \cite{Bt2}.

%
\end{Emp}


\begin{Emp} \label{E:stratum}
{\bf Fibration over a GKM stratum.} Fix a GKM pair $(w,\br)$.

\smallskip

(a) 
As in \cite[Sections~4.1.5--4.1.6]{BKV}, we have a canonical embedding $T_{w,\br}\hra\kC_{w,\br}$, unique up to an $\cL G$-conjugacy, thus a canonical map $\psi_{w,\br}:T_{w,\br}\to[\kC_{w,\br}/\cL G]_{\red}$. Furthermore, arguing as in \cite[Corollaries~4.1.10 and 4.1.12]{BKV} one concludes that the map $\psi_{w,\br}$ is a smooth
covering, $[\kC_{w,r}/\cL G]_{\red}$ is an admissible $\infty$-stack (see Section~\re{admstack}), and  $\psi_{w,\br}$ induces an isomorphism
\[
[\psi_{w,\br}]:[T_{w,\br}/(W_{w,r}\ltimes (\cL T_{w})_{\red}]\isom[\kC_{w,\br}/\cL G]_{\red},
\]
where the group scheme $(\cL T_{w})_{\red}$ acts on $T_{w,\br}$ trivially. 
\smallskip

(b) Consider the Cartesian diagram
\begin{equation} \label{Eq:basic1}
\begin{CD}
{X}_{w,\br} @>{\phi}_{w,\br}>> [\wt{\fC}_{w,\br}/\cL G]_{\red} \\
 @V{{g}_{w,\br}}VV @VV\ov{\frak{p}}_{w,\br}V\\
T_{w,\br} @>\psi_{w,\br}>> [\fC_{w,\br}/\cL G]_{\red}.
\end{CD}
\end{equation}
By definition, ${X}_{w,\br}$ is the reduction of a closed ind-subscheme
\[
\{([g],\g)\in\Fl\times T_{w,\br}~\vert~g^{-1}\g g\in I\}\subseteq \Fl\times T_{w,\br}.
\]
In particular, there is a natural action of the group scheme $(\cL T_w)_{\red}$ on ${X}_{w,\br}$ over $T_{w,\br}$, given by the formula
$t([g],\g):=([tg],\g)$ (compare \cite[Section~4.3.2]{BKV}).

\smallskip

(c) Recall that group $\La_w:=\Hom_F(\gm,T_w)$ is naturally a subgroup of $\cL(T_w)_{\red}$ via the embedding $\la\mapsto\la(t)$. Therefore the action of $\cL(T_w)_{\red}$ from part~(b) induces an action of $\La_w$ on $X_{w,\br}$ over $T_{w,\br}$. Moreover, the composition
\[
{X}_{w,\br}\overset{\phi_{w,\br}}{\lra}[\wt{\fC}_{w,\br}/\cL G]\overset{\pr}{\lra}T:([g],\g)\mapsto[g^{-1}\g g]
\]
is $\La_w$-equivariant and thus factors through the projection ${X}_{w,\br}\to [{X}_{w,\br}/\La_w]$.
\end{Emp}

The following result is a group analog of \cite[Corollary~4.3.4]{BKV}. 

\begin{Prop} \label{P:topproper}
(a) In the notation of Section~\re{stratum}, ${X}_{w,\br}$ is a placid reduced scheme locally finitely presented over $T_{w,\br}$, equipped with an action of
$\La_w$. Moreover, the quotient $[{X}_{w,\br}/\La_w]$ is an algebraic space, fp-proper over $T_{w,\br}$. Furthermore, there exists a subgroup $\La'_w\subseteq \La_w$ of finite index such that the quotient 
$[{X}_{w,\br}/\La'_w]$ is a scheme.

\smallskip

(b) The map $\ov{\frak{p}}_{w,\br}$ is a locally fp-representable morphism between placid $\infty$-stacks. Moreover, it is
uo-equidimensional of relative dimension $\dt_{w,\br}$.
\end{Prop}

\begin{proof}
To prove the results, we repeat the arguments of \cite[Corollary~4.3.4(a),(b)]{BKV}, replacing \cite[Theorem~3.4.7]{BKV} by Theorem~\ref{flat1} and replacing dimension formula of Bezrukavnikov \cite{Be} by its group version (see Section~\re{affsprfib}(b)).
\end{proof}

\begin{Thm} \label{small}
The projection $\ov{\frak{p}}_{\bullet}:[\wt{\fC}_{\bullet}/\cL G]\to[\kC_{\bullet}/\cL G]$ is $[\kC_{\leq 0}/\cL G]$-small (see Section~\re{small}).
\end{Thm}
\begin{proof}
Note that $([\kC_{\bullet}/\cL G],\{[\kC_{w,\br}/\cL G]_{\red}\}_{w,\br})$ is a placidly stratified $\infty$-stack (by
Sections~\re{open}(c) and \re{stratum}(a)), and that $[\wt{\fC}_{\bullet}/\cL G]\simeq[I_{\bullet}/I]$ is a placid $\infty$-stack
(by Section~\re{fibr}(c) and \re{basic example}(a)).

\smallskip

Moreover, each $[\wt{\fC}_{w,\br}/\cL G]\simeq [I_{w,\br}/I]$ is an fp-locally closed substack of $[\wt{\fC}_{\bullet}/\cL G]\simeq[I_{\bullet}/I]$ of pure codimension $b_{w,\br}$ (by \rl{codim} and Section~\re{dimfn}(b)(ii)), each $\ov{\frak{p}}_{w,\br}$ is locally fp-representable and equidimensional of relative dimension $\delta_{w,\br}$
(by \rp{topproper}(b) and Section~\re{eqdim}(b)). Finally, inequalities of Section~\re{small}(c)(iii) follow from Corollary~\ref{cor-codim}.
\end{proof}

\begin{Emp} \label{E:regstratum}
{\bf Fibration over a regular stratum.}

\smallskip

(a) Arguing as in \cite[Corollaries~4.2.2--4.2.3]{BKV}, we have a commutative diagram
\[
\begin{CD}
[\clp(T^{\rs})/\clp(T)] @>\sim>> [\wt{\fC}_{\leq 0}/\cL G] \\
 @VVV @VV\ov{\frak{p}}_{\leq 0}V\\
[\clp(T^{\rs})/W\ltimes\cL(T)_{\red}] @>\sim>> [\fC_{\leq 0}/\cL G]_{\red},
\end{CD}
\]
where the top isomorphism is the isomorphism $[\psi_{w,\br}]$ of Section~\re{stratum}(a) for $(w,\br)=(1,0)$.

\smallskip

(b) Using isomorphism $\cL(T)_{\red}\simeq \clp(T)\times\La$ (see \cite[Section~4.1.11]{BKV}), we conclude from the commutative diagram of part~(a),
that the projection $\ov{\frak{p}}_{\leq 0}:[\wt{\fC}_{\leq 0}/\cL G]\to [\fC_{\leq 0}/\cL G]_{\red}$ is a $\wt{W}$-torsor.

\smallskip

(c) Notice that the composition
\[
[\clp(T^{\rs})/\clp(T)]\isom [\wt{\fC}_{\leq 0}/\cL G]\overset{\pr}{\lra}T
\]
sends $[t]\mapsto\ev(t)$. In particular, it is $\wt{W}$-equivariant with respect to the natural action of $W$ and the trivial action of $\La$ on both sides.
\end{Emp}

\section{Application to affine Grothendieck--Springer sheaves}

\subsection{Perversity and $\wt{W}$-action}


\begin{Emp}
{\bf Observations} (compare \cite[Sections~7.1.1--7.1.3]{BKV}).

\smallskip

(a) Recall the projection $\ov{\frak{p}}:[\wt{\fC}/\cL G]\to[\kC/\cL G]$
is ind-fp-proper (see Section~\re{fibr}(c)), so the pullback $\ov{\frak{p}}^{!}$ admits a left adjoint $\ov{\frak{p}}_!$ that satisfies base change (by \cite[Proposition~5.3.7]{BKV}).

\smallskip

(b) Using observations of Section~\re{open}(b) and arguing as in \cite[Lemma~7.1.2]{BKV}, one obtains that
the $\infty$-stack $[\kC_{\bullet}/\cL G]$ admits gluing of sheaves.

\smallskip

(c) By Theorem~\ref{small}, the stratified $\infty$-stack $([\fC_{\bullet}/\cL G],\{[\fC_{w,\br}/\cL G]_{\red}\}_{w,\br})$ is placidly stratified, and the projection $\ov{\frak{p}}_{\bullet}:[\wt{\fC}_{\bullet}/\cL G]\to[\kC_{\bullet}/\cL G]$ is $[\kC_{\leq 0}/\cL G]$-small.
Therefore morphism $\ov{\frak{p}}_{\bullet}$ gives rise to a perversity $p_{\nu}=\{\nu_{w,\br}\}_{w,\br}$,
defined by $\nu_{w,\br}:=\dt_{w,\br}+b_{w,\br}$  (see Sections~\re{pervfun}(b) and \re{notcodim}).

\smallskip

(d) By parts~(b),(c) and Section~\re{pervplstr}(b), the $\infty$-category $\cD([\fC_{\bullet}/\cL G])$ is equipped with
 a canonical $t$-structure, corresponding to the perversity $p_{\nu}$.
\end{Emp}

\begin{Thm} \label{T:perv}
For every $!$-local system $K\in\Loc^{!}([\wt{\fC}_{\bullet}/\cL G])$, the push-forward $(\ov{\frak{p}}_{\bullet})_!(K)$ is $p_{\nu}$-perverse. Moreover, it is the intermediate extension of its restriction to $[\kC_{\leq 0}/\cL G]$.
\end{Thm}
\begin{proof}
Since projection $\ov{\frak{p}}_{\bullet}$ is $[\kC_{\leq 0}/\cL G]$-small (by Theorem~\ref{small}), the assertion follows from
Theorem~\ref{g-small}.
\end{proof}

\begin{Emp} \label{E:gsp}
{\bf Basic example.}

\smallskip

(a) Let $\pr:[\wt{\fC}/\cL G]\to T$ be the projection (see Section~\re{fibr}(d)). Then a local system $\cL$ on $T$ gives rise to the $!$-local system $\om_{\cL}$ on $[\wt{\fC}/\cL G]$ (see Section~\re{!ls}(b)), hence to the object  $\cS_{\cL}:=\ov{\frak{p}}_!(\om_{\cL})\in\cD([\kC/\cL G])$, called the {\em affine Grothendieck--Springer sheaf}.

\smallskip

(b) Let $\cS_{\cL,\bullet}$ be the restrictions of $\cS_{\cL}$ to the open substack $[\kC_{\bullet}/\cL G]$. Since $\ov{\frak{p}}_!$ satisfies base change, we have an identification $\cS_{\cL,\bullet}\simeq (\ov{\frak{p}}_{\bullet})_!(\om_{\cL})\in\cD([\kC_{\bullet}/\cL G])$. Thus, by \rt{perv}, the sheaf $\cS_{\cL,\bullet}$ is $p_{\nu}$-perverse and is the intermediate extension of its restriction $\cS_{\cL,\leq 0}\in\cD([\kC_{\leq 0}/\cL G])$.

\smallskip

(c) For every $\g\in \kC_{\bullet}(k)$, we denote by $\iota_{\g}$ the corresponding morphism $\pt\to [\kC_{\bullet}/\cL G]$.
Then, by base change, we have an isomorphism
\begin{equation} \label{Eq:rgmc}
\iota_{\g}^!(\cS_{\cL,\bullet})\simeq R\Gm_c(\Fl_{\g},\om_{\cL}).
\end{equation}
\end{Emp}

\begin{Cor} \label{C:action}
(a) For a local system $\cL$ on $T$, we have natural isomorphisms 
\[
a_{w,\cL}:\cS_{\cL,\bullet}\isom \cS_{w_!(\cL)\bullet}\text{ for } w\in \wt{W},\text{ satisfying }a_{w_1,(w_2)_!(\cL)}\circ a_{w_2,\cL}\simeq a_{w_1w_2,\cL}\text{ for }w_1,w_2\in\wt{W}.
\]

\smallskip

(b) The affine Grothendieck--Springer sheaf $\cS_{\cL,\bullet}\in\cD([\kC_{\bullet}/\cL G])$ is equipped with a natural $\La$-action, that is,   has a natural lift to $\cD^{\La}([\kC_{\bullet}/\cL G])$. Furthermore, $\cS_{\cL,\bullet}$ has a natural lift to $\cD^{\wt{W}}([\kC_{\bullet}/\cL G])$,
if $\cL$ is $W$-equivariant.
\end{Cor}

\begin{proof}
(a) Since $\cS_{\cL,\bullet}$ is the intermediate extension of $\cS_{\cL,\leq 0}$ (see Section~\re{gsp}(b)), for every $w\in\wt{W}$ the restriction map
\[
\Hom(\cS_{\cL,\bullet},\cS_{w_!(\cL),\bullet})\to \Hom(\cS_{\cL,\leq 0},\cS_{w_!(\cL),\leq 0})
\]
is an isomorphism. Thus, it remains to construct a collection of isomorphisms
\[
a_{w,\cL}:\cS_{\cL,\leq 0}\isom \cS_{w_!(\cL),\leq 0},\text{ satisfying  }a_{w_1,(w_2)_!(\cL)}\circ a_{w_2,\cL}\simeq a_{w_1w_2,\cL}.
\]

By the base change, we have an isomorphism  $\cS_{\cL,\leq 0}\simeq (\ov{\frak{p}}_{\leq 0})_!(\om_{\cL})$.
Recall that the projection $\ov{\frak{p}}_{\leq 0}:[\wt{\kC}_{\leq 0}/\cL G]\to [\kC_{\leq 0}/\cL G]_{\red}$ is a $\wt{W}$-torsor, and the projection $\pr:[\wt{\kC}_{\bullet}/\cL G]\to T$ is $\wt{W}$-equivariant (see Section~\re{regstratum}(b),(c)).  Since $\om_{\cL}=\pr^!(\om_T\otimes \cL)\in \cD([\wt{\kC}_{\leq 0}/\cL G])$,
for every $w\in\wt{W}$ we have an isomorphism $w_!(\om_{\cL})\simeq \om_{w_!(\cL)}$, hence an isomorphism
\[
a_{w,\cL}: \cS_{\cL,\leq 0}\simeq (\ov{\frak{p}}_{\leq 0})_!(\om_{\cL})\simeq (\ov{\frak{p}}_{\leq 0}\circ w)_!(\om_{\cL})\simeq
(\ov{\frak{p}}_{\leq 0})_!w_!(\om_{\cL})\simeq (\ov{\frak{p}}_{\leq 0})_!(\om_{w_!(\cL)})\simeq \cS_{w_!(\cL),\leq 0}.
\]
Moreover, the identity $a_{w_1,(w_2)_!(\cL)}\circ a_{w_2,\cL}\simeq a_{w_1w_2,\cL}$ is straightforward.

\smallskip

(b) Since perverse sheaves have no negative Exts, it follows for example from  \cite[Remark~1.2.1.12]{Lu2} that the $\infty$-category
$\Perv^{p_{\nu}}([\kC_{\bullet}/\cL G])$ of $p_{\nu}$-perverse sheaves is equivalent to its homotopy category. Since each $\cS_{\cL,\bullet}$ is $p_{\nu}$-perverse, it thus enough to construct the action at the level of homotopy categories. So both assertions follows from part~(a).
\end{proof}

\begin{Cor}\label{C:action2}
Fix $\g\in\fC_{\bullet}(k)$. For a local system (resp. $W$-equivariant local system) $\cL$ on $T$,

\smallskip

(a) the complex $R\Gm_c(\Fl_{\g},\omega_{\cL})$ is equipped with an action of $\La\times\cL G_{\g}$ (resp. $\wt{W}\times\cL G_{\g}$);

\smallskip

(b) for every $i\in\B{Z}$, the cohomology group  $H^i_c(\Fl_{\g},\omega_{\cL})$ is equipped with an action of
$\La\times\pi_0(\cL G_{\g})$ (resp. $\wt{W}\times\pi_0(\cL G_{\g})$). 
\end{Cor}

\begin{proof}
(a) Notice that the morphism $\iota_{\g}:\pt\to [\kC_{\bullet}/\cL G]$ from Section~\re{gsp}(c) factors through $[\pt/\cL G_{\g}]$.
Therefore the pullback $\iota_{\g}^!(\cS_{\cL,\bullet})\in \cD^{\La}(\pt)$ of $\cS_{\cL,\bullet}\in\cD^{\La}([\kC_{\bullet}/\cL G])$ has a natural lift to an object of $\cD^{\La}([\pt/\cL G_{\g}])=\cD^{\La\times\cL G_{\g}}(\pt)$. From this the assertion follows using isomorphism \form{rgmc}. The assertion for $W$-equivariant local systems is similar.

\smallskip

(b) By part~(a), for every $i\in\B{Z}$, the cohomology group  $H^i_c(\Fl_{\g},\omega_{\cL})$ is equipped with an action of
$\La\times\cL G_{\g}$ (resp. $\wt{W}\times\cL G_{\g}$). Now the assertion follows from the fact that the $\cL G_{\g}$-action on each $H^i_c(\Fl_{\g},\omega_{\cL})$ factors through
$\pi_0(\cL G_{\g})$.
\end{proof}

\begin{Emp}\label{E:hom}
{\bf Remark.} By the Poincar\'e duality, the cohomology group $H^i_c(\Fl_{\g},\omega_{\cL})$ coincides with the homology group $H_{-i}(\Fl_{\g},\C{F}_{\cL})$ considered in \cite{BV}. Thus, by \cite[Corollary~2.2.8]{BV}, we have an a priori different Lusztig action
of  $\La$ (resp. $\wt{W}$) on each $H^i_c(\Fl_{\g},\omega_{\cL})$.
\end{Emp}

\begin{Prop} \label{P:compat}
For every $\g\in\fC_{\bullet}(k)$ and $i\in\B{Z}$, the action of $\La$ (resp. $\wt{W}$) on $H^i_c(\Fl_{\g},\omega_{\cL})$ from \rco{action2}
coincides with the action from \cite[Corollary~2.2.8]{BV}.
\end{Prop}

\begin{proof}
By the construction of the actions in \rco{action2} and \cite[Corollary~2.2.8]{BV}, it suffices to show that for every $w\in\wt{W}$,
the isomorphism
\begin{equation} \label{Eq:isom}
a_{w,\cL}:H^i_c(\Fl_{\g},\omega_{\cL})\isom H^i_c(\Fl_{\g},\omega_{w_!(\cL)}),
\end{equation}
induced by the isomorphism
$a_{w,\cL}$ of \rco{action}, coincides with the isomorphism of \cite[Proposition~2.2.7(c)]{BV}.

\smallskip

Recall that the affine Weyl group $\wt{W}$ is generated by a set $\wt{S}\subseteq\wt{W}$ of simple affine reflections and a finite abelian group
$\Omega:=N_{\cL G}(I)/I$ (compare \cite[Section~2.1.2(d)]{BV}). Thus, it suffices to show the equality of isomorphisms $a_{w,\cL}$ for  $w\in\Omega$ and $w\in\wt{S}$.

\smallskip

To show the assertion for $w\in\Omega=N_{\cL G}(I)/I$, note that $w$ induces an automorphism $w_*$ of
$[\wt{\fC}/\cL G]\simeq [I/I]$ over $[\fC/\cL G]$ and in both cases the isomorphism $a_{w,\cL}$ of \form{isom} is induced by $w_*$.

\smallskip

To show the assertion for $w=s\in\wt{S}$, consider diagram of adjoint quotients
\[
\begin{CD}
[I/I] @>\pr^s>> [B_s/B_s] @>\pr_s>> T\\
@V\ov{\frak{p}}_s VV @V p_s VV\\
[P_s/P_s] @>\ov{\pr}^s>> [L_s/L_s]\\
@V\ov{\frak{p}}^s VV\\
[\fC/\cL G],
\end{CD}
\]
where $P_s\supsetneq I$ is the minimal standard parahoric of $\cL G$, corresponding to $s$, $P_s^+\subseteq P_s$ is the pro-unipotent radical of $P_s$, $L_s:=P_s/P^+_s$ is the corresponding ``Levi subgroup'', $B_s:=I/P_s^+\subseteq L_s$ is the Borel subgroup of $L_s$, and
all the maps are natural projections. 

\smallskip

Set $\cS^{\on{fin}}_{\cL}:=(p_s)_{!}\pr_s^!(\om_{T,\C{L}})\in\cD([L_s/L_s])$. Then, by the usual (finite-dimensional) Springer theory (see, for example, \cite[Section~1.2.1.(e)]{BV}),
we have a natural isomorphism
\[
a^{\on{fin}}_{s,\cL}: \cS^{\on{fin}}_{\cL}\isom\cS^{\on{fin}}_{s_!(\cL)}.
\]

On the other hand, by the proper base change, we have a natural identification
\[
\cS_{\cL}\simeq (\ov{\frak{p}}^s)_!(\ov{\frak{p}}_s)_!(\pr^s)^!\pr_s^!(\om_{T,\C{L}})\simeq
(\ov{\frak{p}}^s)_!(\ov{\pr}^s)^!(p_s)_{!}\pr_s^!(\om_{T,\C{L}})=(\ov{\frak{p}}^s)_!(\ov{\pr}^s)^!(\cS^{\on{fin}}_{\cL}).
\]
So isomorphism $a^{\on{fin}}_{s,\cL}$ induces an isomorphism
\begin{equation*} \label{Eq:awl}
a^{\on{Lu}}_{s,\cL}: \cS_{\cL}\isom\cS_{s_!(\cL)},
\end{equation*}
where ``$\on{Lu}$'' stands for Lusztig, and by definition the isomorphism \form{isom} of \cite[Proposition~2.2.7(c)]{BV} is induced by isomorphism $a^{\on{Lu}}_{s,\cL}$.

\smallskip

It remains to check that the restriction of the isomorphism $a^{\on{Lu}}_{s,\cL}$ to $[\fC_{\bullet}/\cL G]$ coincides with the isomorphism
of $a_{s,\cL}$ of \rco{action}. By construction, it suffices to show that the restriction of the isomorphism
$a^{\on{Lu}}_{s,\cL}$ to $[\fC_{\leq 0}/\cL G]$ is induced by the geometric action of $s\in \wt{W}$ on $[\wt{\fC}_{\leq 0}/\cL G]$. But this follows from the fact that the usual Springer action $a^{\on{fin}}_{s,\cL}$ is induced by the geometric action on the regular semisimple locus.
\end{proof}

\begin{Emp}
{\bf Remarks.} The following observations are not used in this work:

\smallskip

(a) Arguing as in \cite[Proposition~2.2.7]{BV} or \cite[Section~3.3]{B} (using presentation of $\wt{W}$ by generators and relations), one can show that a small modification of the argument of \rp{compat} provides a Lusztig action of $\La$ (resp. $\wt{W}$) on $\cS_{\cL}$ on the level of homotopy categories. In particular, using the perversity of $\cS_{\cL,\bullet}$ and observation of \cite[Remark~1.2.1.12]{Lu2} again, one can deduce that the induced action on $\cS_{\cL,\bullet}$ lifts uniquely to the action on the level of $\infty$-categories. Moreover, the resulting action of $\La$ (resp. $\wt{W}$) on  $\cS_{\cL,\bullet}$ coincides  with the one of \rco{action}.

\smallskip

(b) Furthermore, using perversity of finite-dimensional Springer sheaves, one can show that the Lusztig action of $\La$ (resp. $\wt{W}$) on $\cS_{\cL}$ on the level of homotopy categories discussed in part~(a) can be lifted to the action on the level of $\infty$-categories (see \cite{BeKV2}). Furthermore, this lift can be shown to be unique.
\end{Emp}

\begin{Cor} \label{C:fingen}
For every $\g\in\fC_{\bullet}(k)$ and $i\in\B{Z}$, the cohomology group $H^i_c(\Fl_{\g},\omega_{\cL})$
is a finitely generated $\qlbar[\La]$-module.
\end{Cor}

\begin{proof}
This follows from a combination of \rp{compat} and \cite[Proposition~3.3.2]{BV}.
\end{proof}

\begin{Emp}
{\bf Remark.} The proof of \cite[Proposition~3.3.2]{BV} uses a group version \cite[Proposition~2.3.4]{BV} of Yun's theorem \cite{Yun} on the compatibility of actions on homology of affine Springer fibers, whose proof is global. A different (purely local) proof of this assertion will appear in \cite{BeKV2}.

\end{Emp}

\begin{Thm} \label{T:La-cstr}
For every local system $\cL$ on $T$, the sheaf $\cS_{\cL,\bullet}\in \cD^{\La}([\kC_{\bullet}/\cL G])$ is $\La$-constructible.
\end{Thm}

\begin{proof}
By \rl{Gmconsstr}(a), it suffices to show that the $!$-restriction of $\cS_{\cL,\bullet}$ to each  stratum $[\kC_{w,\br}/\cL G]_{\red}$ is
$\La$-constructible. Moreover, since $\psi_{w,\br}:T_{w,\br}\to [\kC_{w,\br}/\cL G]_{\red}$ is a covering (see Section~\re{stratum}(a)), it follows from \rl{Gmconsstr}(b) that it suffices to show that the $!$-pullback  $(\psi_{w,\br})^!(\cS_{\cL,\bullet})\in \cD(T_{w,\br})$ is $\La$-constructible.

\smallskip

By the base change, we have $(\psi_{w,\br})^!(\cS_{\cL,\bullet})\simeq(g_{w,\br})_!(\omega_{\cL})$, where
$g_{w,\br}:X_{w,\br}\to T_{w,\br}$ is the morphism from Section~\re{stratum}(b). Moreover, as it is shown in \rp{topproper}, the group $\La_{w}$ acts on $X_{w,\br}$ over $T_{w,\br}$, and the quotient $[X_{w,\br}/\La_w]$ is an fp-proper algebraic space over $T_{w,\br}$.

Since the composition $\pr\circ \phi_{w,\br}:X_{w,\br}\to T$ factors through the projection $X_{w,\br}\to[X_{w,\br}/\La_w]$
(see Section~\re{stratum}(c)), it thus follows from \rl{La-cstr} that $(g_{w,\br})_!(\omega_{\cL})\in\cD(T_{w,\br})$ has a natural lift to $\cD^{\La_w}(T_{w,\br})$, and the corresponding lift is $\La_w$-constructible. Hence,
by Section~\re{func}(b), the sheaf $(g_{w,\br})_!(\omega_{\cL})$ is essentially constructible.

\smallskip

Note that $T_{w,\br}$ is an admissible scheme (see Section~\re{constr}(b)) and algebra $\bql[\La]$ is Noetherian of finite cohomological dimension. Thus, by \rp{Gmcons},  it suffices to show that for every $\g\in T_{w,\br}(k)$, the pullback
\[
\iota_{\g}^{!}((g_{w,\br})_!(\omega_{\cL}))\simeq R\Gm_c(\Fl_{\g},\omega_{\cL})\in\cD^{\La}(\pt)
\]
is a perfect complex. Equivalently, we have to show that each cohomology group
$H^i_c(\Fl_{\g},\omega_{\cL})$ is a finitely-generated $\qlbar[\La]$-module. But this follows from \rco{fingen}.
\end{proof}

\begin{Emp}
{\bf Remark.} A different proof of (a generalization of) \rt{La-cstr} will appear in \cite{BeKV2}.
\end{Emp}

\begin{Cor} \label{C:ess-cons}
For every local system $\cL$ on $T$ and representation $\tau\in\on{Rep}_{\bql}(\La)$, the sheaf of $\tau$-coinvariants $\cS_{\cL,\bullet,\tau}:=\on{coinv}_{\tau}(\cS_{\cL,\bullet})$ is essentially constructible. Furthermore, $\cS_{\cL,\bullet,\tau}$ is constructible, if $\tau$ is finite dimensional.
\end{Cor}

\begin{proof}
Since $\cS_{\cL,\bullet}\in \cD^{\La}([\kC_{\bullet}/\cL G])$ is $\La$-constructible (by \rt{La-cstr}), the assertion follows from \rl{Gmconsstr}(c).
\end{proof}

\subsection{Perversity of $\tau$-coinvariants}


Now we are ready to prove the main result of this work.

\begin{Thm}\label{T:perv-tau}
For every local system on $\cL$ on $T$ and representation $\tau\in\on{Rep}_{\bql}(\La)$, the sheaf of $\tau$-coinvariants $\cS_{\cL,\bullet,\tau}=\on{coinv}_{\tau}(\cS_{\cL,\bullet})\in\cD([\kC_{\bullet}/\cL G])$ is $p_{\nu}$-perverse.
\end{Thm}

\begin{Emp}
{\bf Remark.}
Unlike $\cS_{\cL,\bullet}$, it is not true that $\cS_{\cL,\bullet,\tau}$ is the intermediate extension of its restriction to
$[\kC_{\leq 0}/\cL G]$. Also $\cS_{\cL,\bullet,\tau}$ is not necessary irreducible, even if $\tau$ and $\C{L}$ are irreducible.
\end{Emp}

\begin{proof}
As $\cS_{\cL,\bullet}$ is $p_{\nu}$-perverse (by \rt{perv}), and the functor of $\tau$-coinvariants is right $t$-exact (by Lemma~\ref{r-tex}), we only have to prove that $\cS_{\cL,\bullet,\tau}\in \mathstrut^{p_{\nu}}\cD^{\geq 0}([\kC_{\bullet}/\cL G])$. By definition, we have to check that for every GKM pair $(w,\br)$, we have
\begin{equation*}
\eta_{w,\br}^{!}(\cS_{\C{L},\bullet,\tau})\in \mathstrut^{p}\cD^{\geq -\nu_{w,\br}}([\kC_{w,\br}/\cL G]_{\red}).
\label{res-tau}
\end{equation*}

Next, using identity $\nu_{w,\br}=\dt_{w,\br}+b_{w,\br}=2\dt_{w,\br}+c_w+a_{w,\br}=d_{\br}+a_{w,\br}$ (by Proposition~\ref{codim}),
it thus suffices to check that for every GKM pair $(w,\br)$, we have
\begin{equation*}
\eta_{w,\br}^{!}(\cS_{\C{L},\bullet,\tau})\in \mathstrut^{p}\cD^{\geq -d_{\br}}([\kC_{w,\br}/\cL G]_{\red}).
\end{equation*}

Since $\cS_{\cL,\bullet,\tau}$ is essentially constructible (by \rco{ess-cons}), the pullback $\eta_{w,\br}^{!}(\cS_{\C{L},\bullet,\tau})$ is essentially constructible (by Section~\re{esscons}(c)). Using the fact that the $\infty$-stack $[\kC_{w,\br}/\cL G]_{\red}$ is admissible (see Section~\re{stratum}(a)), it follows from Proposition~\ref{pullcheck}
that it suffices to show that for every $\g\in \kC_{w,r}(k)$, one has an inclusion
\begin{equation} \label{Eq:res-tau}
\iota_{\g}^{!}(\cS_{\C{L},\bullet,\tau})\in \cD^{\geq -d_{\br}}(\pt).
\end{equation}
Combining isomorphism $\iota_{\g}^{!}(\cS_{\C{L},\bullet})\simeq R\Gm_c(\Fl_{\g},\omega_{\cL})$ (see Section~\re{gsp}(c)) with observations of Sections~\re{isot}(c),(d), we get an isomorphism
\[
\iota_{\g}^{!}(\cS_{\C{L},\bullet,\tau})\simeq (\iota_{\g}^{!}(\cS_{\C{L},\bullet}))_{\tau}\simeq R\Gm_c(\Fl_{\g},\omega_{\cL})_{\tau}\simeq \tau\otimes^{L}_{\bql[\La]} R\Gm_c(\Fl_{\g},\omega_{\cL}),
\]
where $\tau$ is the right $\qlbar[\La]$-module, corresponding to $\tau\in\on{Rep}_{\bql}(\La)$ (see Section~\re{isot}(d)). Now inclusion \form{res-tau} follows from Proposition~\ref{p-tau2} below.
\end{proof}


\begin{Prop}\label{p-tau2}
For every $\g\in\kC_{w,\br}(k)$ and every representation $\tau\in\on{Rep}_{\bql}(\La)$, we have
\[\tau\otimes^{L}_{\bql[\La]}R\Gm_c(\Fl_{\g},\omega_{\cL})\in\cD^{\geq -d_{\br}}.\]
\end{Prop}

\begin{proof}
Set $V:=R\Gm_c(\Fl_{\g},\omega_{\cL})$ and we want to show that $\tau\otimes^{L}_{\bql[\La]}V\in\cD^{\geq -d_{\br}}$.
To make the proof more structural, we will divide it into steps:

\smallskip

{\bf Step 1.} Note that $\La_{\g}:=\Hom_F(\B{G}_m,G_{\g})$ is naturally a subgroup of $\cL G_{\g}$ via the map $\la\mapsto \la(t)$.
By \rco{action2}(a), $V$ is equipped with an action of $\La\times \cL G_{\g}$, therefore with an action of
$\La\times \La_{\g}$. We claim that for every $\nu\in\on{Rep}_{\qlbar}(\La_{\g})$, we have
\[
\nu\otimes^{L}_{\bql[\La_{\g}]}V\in\cD^{\geq -2\dt_{w,\br}}.
\]

\begin{proof}
Consider projection $p:\Fl_{\g}\to Y:=[\Fl_{\g}/\La_{\g}]$, and let $\pr_{\g}:\Fl_{\g}\to T$ be the restriction of $\pr:[\wt{\fC}/\cL G]\to T$ (see Section~\re{fibr}(d)). As in Section~\re{stratum}(c), the projection $\pr_{\g}$ factors through
$p:\Fl_{\g}\to Y$. Thus, $V\simeq R\Gm_c(Y,p_!p^!\omega_{\cL})$, where we write $\omega_{\cL}$ instead of $\omega_{Y,\cL}$ (see Section~\re{!ls}(b)). Moreover, the $\La_{\g}$-action on $V$ comes from the $\La_{\g}$-action on $p_!p^!\omega_{\cL}$ (compare \rl{taucoinv}(a)).

\smallskip

Then, using Sections~\re{isot}(c),(d) and \rl{taucoinv}(b), we have an identification

\[
\nu\otimes^{L}_{\bql[\La]}V\simeq \coinv_{\nu}(V)\simeq\coinv_{\nu} (R\Gm_c(Y,p_!p^!\omega_{\cL}))\simeq
\]
\[
\simeq R\Gm_c(Y,\coinv_{\nu}(p_!p^!\omega_{\cL})))\simeq  R\Gm_c(Y,A_{\nu}\overset{!}{\otimes} \omega_{\cL}).
\]

Using Sections~\re{!ls}(b),(c) we conclude that $A_{\nu}\overset{!}{\otimes} \omega_{\cL}$ is a $!$-local system on $Y$, hence it lies in $\cD^{\geq -2\dim Y}(Y)$ with respect to the usual (rather than perverse) $t$-structure.
%
Therefore we have
\[
\nu\otimes^{L}_{\bql[\La]}V\simeq  R\Gm_c(Y,A_{\nu}\overset{!}{\otimes}\omega_{\cL})\in D^{\geq -2\dim Y},
\]
so our assertion follows from the equality $\dim Y=\dim\Fl_{\g}=\dt_{w,\br}$.
\end{proof}

\smallskip

{\bf Step 2.} As in \cite[Section~2.3.1]{BV}, we have a canonical isomorphism $\pi_{0}(\cL G_{\g})\simeq X_*(G_{\g})_{\Gm_{F}}$, where we set
$X_*(G_{\g}):=\Hom_{\ov{F}}(\gm,G_{\g})$. Moreover the composition $\La_{\g}\hra \cL G_{\g}\to\pi_0(\cL G_{\g})$ 
can be rewritten as
\[
\La_{\g}=X_*(G_{\g})^{\Gm_{F}}\hra X_*(G_{\g})\to X_*(G_{\g})_{\Gm_{F}}\simeq\pi_0(\cL G_{\g}).
\]
In particular, the homomorphism $\La_{\g}\to\wt{\La}_{\g}:=\pi_0(\cL G_{\g})$ is injective, and the quotient $\ov{\La}_{\g}:=\wt{\La}_{\g}/\La_{\g}$ is finite.

\smallskip

Consider induced representation  $\wt{V}:=\on{ind}_{\La_{\g}}^{\wt{\La}_{\g}}(V)$ of $\La\times \wt{\La}_{\g}$. Then, by Step 1, for every representation  $\wt{\nu}\in\on{Rep}_{\qlbar}(\wt{\La}_{\g})$, we have
\[
\wt{\nu}\otimes^{L}_{\bql[\wt{\La}_{\g}]}\wt{V}\simeq (\wt{\nu}|_{\La_{\g}})\otimes^{L}_{\bql[\La_{\g}]}V\in \cD^{\geq -2\dt_{w,\br}},
\]
where $\wt{\nu}|_{\La_{\g}}$ denotes the restriction of $\nu$, and we want to show that
\[
\tau\otimes^{L}_{\bql[\La]}\wt{V}\simeq (\tau\otimes^{L}_{\bql[\La]}V)^{\oplus|\ov{\La}_{\g}|}\in \cD^{\geq -d_{\br}}.
\]

\smallskip

{\bf Step 3.} Next we are going to rephrase the assertion of Step 2 geometrically:

\smallskip

Set $X:=\Spec(\bql[\wt{\La}_{\g}])$ and $Y:=\Spec(\bql[\La])$. Then $X$ and $Y$ are connected affine algebraic groups (hence connected
smooth affine schemes) of dimensions
$\dim X=r-c_{w}$ and $\dim Y=r$. Since $\wt{V}$ is equipped with an action of
$\La\times\wt{\La}_{\g}$, it corresponds to an object $K\in D^b(\on{QCoh}(X\times Y))$.

\smallskip

Note that a quasi-coherent sheaf $A\in\on{QCoh}(X)$ corresponds to a representation $\wt{\nu}\in\on{Rep}_{\bql}(\wt{\La}_{\g})$, and the derived tensor product $\wt{\nu}\otimes^{L}_{\bql[\wt{\La}_{\g}]}\wt{A}$ corresponds to $K\otimes^{L} p_{X}^{*}(\iota^*A)\in D^{b}(\on{QCoh}(X\times Y))$, where involution $\iota:\la\mapsto\la^{-1}$ appears because of conventions of Section~\re{isot}(d).
 Similarly, a representation $\tau\in \on{Rep}_{\bql}(\La)$ corresponds to a quasi-coherent sheaf $B\in\on{QCoh}(Y)$ and the derived tensor product $K\otimes^{L} p_{Y}^{*}(\iota^*B)\in D^{b}(\on{QCoh}(X\times Y))$ corresponds to $\tau\otimes^{L}_{\bql[\La]}\wt{V}$.

\smallskip

Hence, by Step 2, we have $K\otimes^L p_{X}^{*}A\in D^{\geq -2\dt_{w,\br}}$ for all $A\in\on{QCoh}(X)$, and it suffices to show
that we have $K \otimes^{L} p_{Y}^{*}B\in D^{\geq -d_{\br}}$ for all $B\in\on{QCoh}(Y)$.


\smallskip

{\bf Step 4.} Recall that $G_{\g}\subseteq G$ is a maximal torus. Hence every Borel subgroup $B\supseteq G_{\g}$ over $\ov{F}$ gives rise to an isomorphism
\[
\phi=\phi_B:G_{\g}\hra B\to B/[B,B]\simeq T,
\] and following \cite{BV} we call such isomorphisms {\em admissible}.\footnote{The collection of admissible isomorphisms form a $W$-torsor.}

\smallskip

Every admissible isomorphism $\phi$ gives rise to a surjective homomorphism of groups
\[
\La=X_*(T)\overset{\phi}{\lra} X_*(G_{\g}) \to X_*(G_{\g})_{\Gm_F}\simeq \wt{\La}_{\g} \overset{\iota}{\to}\wt{\La}_{\g},
\]
where $\iota$ is the map $\iota:g\mapsto g^{-1}$, hence a surjective homomorphism of $\qlbar$-algebras
$\qlbar[\La]\to\qlbar[\wt{\La}_{\g}]$. Therefore $\phi$ gives rise to a closed embedding $Y\hra X$, which we will denote by $\eta_{\phi}$.


\smallskip

Also the surjective homomorphism of groups $\wt{\La}_{\g}\to\ov{\La}_{\g}$ gives rise to a homomorphism of algebraic groups $\ov{\La}_{\g}=\Spec(\bql[\ov{\La}_{\g}])\to \Spec(\bql[\wt{\La}_{\g}])=X$. In particular, every $\la\in\ov{\La}_{\g}$ induces an automorphism of $X$, hence a closed embedding $\eta_{\phi,\la}:=\la\circ \eta_{\phi}:Y\hra X$.

\smallskip

{\bf Step 5.}  By Claim~\ref{yun-supp} below, the quasi-coherent sheaf $K\in D^b(\on{QCoh}(X\times Y))$ is set-theoretically supported on the union of graphs of the $\eta_{\phi,\la}$'s. Hence, by observation of Step~3 and identities $\dim X-\dim Y=c_w$ and $-d_{\br}=-2\dt_{w,\br}-c_w$, our assertion follows from Proposition~\ref{coh3} applied to the quasi-coherent sheaf $K[-2\dt_{w,\br}]$ and closed embeddings $\{\eta_{\phi,\la}\}_{\phi,\la}$.
\end{proof}

\begin{Cl}\label{yun-supp}
 The quasi-coherent sheaf $K\in D^b(\on{QCoh}(X\times Y))$ from Step 3 is set-theoretically supported on the union of graphs of the $\eta_{\phi,\la}$'s,
 where $\phi$ runs over the set of all admissible isomorphisms $G_{\g}\isom T$ and $\la$ runs over elements of $\ov{\La}_{\g}$.
\end{Cl}

\begin{proof}
It suffices to show that each cohomology sheaf $\cH^i(K)$ is supported on the union of graphs of the $\eta_{\phi,\la}$'s.

\smallskip

By definition, $\cH^i(K)$ corresponds to the representation $H^i(\wt{V})\simeq \on{ind}_{\La_{\g}}^{\wt{\La}_{\g}}(H^i(V))$ of $\La\times \wt{\La}_{\g}$, where (as before) $H^i(V)$ is a representation of $\La\times{\La}_{\g}$, obtained as a restriction of the representation of $\La\times \cL G_{\g}$. Moreover, the action of $\La\times{\La}_{\g}$ on $H^i(V)$ naturally extends to the action of $\La\times \wt{\La}_{\g}$.
Namely, this follows from the fact that the action of $\cL G_{\g}$ on $H^i(V)=H^i(\Fl_{\g},\omega_{\cL})$ factors through
$\wt{\La}_{\g}=\pi_0(\cL G_{\g})$ (see \rco{action2}(b)).

\smallskip

By the observations of the previous paragraph, representation $H^i(\wt{V})$ of $\wt{\La}_{\g}\times\La$ decomposes as a direct sum
$\bigoplus_{\la\in\ov{\La}_{\g}} {}^{\la} H^i(V)$, where ${}^{\la} H^i(V)$ denotes the twist of $H^i(V)$ by $\la$. Therefore, it suffices to show that the quasi-coherent sheaf $K^i\in \on{QCoh}(X\times Y)$, corresponding to $H^i(V)$, is set-theoretically supported on the union of graphs of the $\eta_{\phi}$'s. As it was explained in the proof of \cite[Claim~5.3.4]{BV}, the assertion follows from  \cite[Theorem~2.3.4]{BV}, which is the group version of \cite[Theorem~2]{Yun}.
\end{proof}

\begin{Cor}\label{C:perv-tau}
For every $W$-equivariant local system on $\cL$ on $T$ and every $\tau\in\on{Rep}_{\bql}(\wt{W})$, the sheaf of $\tau$-coinvariants $\cS_{\cL,\bullet,\tau}=\on{coinv}_{\tau}(\cS_{\cL,\bullet})\in\cD([\kC_{\bullet}/\cL G])$ is $p_{\nu}$-perverse.
\end{Cor}

\begin{proof}
Since  $\cS_{\cL,\bullet,\tau}\simeq (\cS_{\cL,\bullet,\tau|_{\La}})^{W}$ and $(\cS_{\cL,\bullet,\tau|_{\La}})^{W}$ is a retract of
$\cS_{\cL,\bullet,\tau|_{\La}}$, the assertion follows from  \rt{perv-tau}.
\end{proof}

\begin{Emp}
{\bf Remark.} Note that \rt{perv-tau} is a formal consequence of \rco{perv-tau}. Indeed, for every local system on $\cL$ on $T$, the local system $\wt{\cL}:=\bigoplus_{w\in W}w_!(\cL)$ has a natural $W$-equivariant structure. Moreover, for every $\tau\in\on{Rep}_{\bql}(\La)$ and $\wt{\tau}:=\ind_{\La}^{\wt{W}}(\tau)\in\on{Rep}_{\bql}(\wt{W})$, we have a natural isomorphism $\cS_{\cL,\bullet,\tau}\simeq \cS_{\wt{\cL},\bullet,\wt{\tau}}$.
\end{Emp}

Finally, we are going to show a version of \rco{perv-tau} for Lie algebras.

\begin{Emp}
{\bf Lie algebra case.}
Let $\fg$ be the Lie algebra of $G$, let $\fC_{\fg,\bullet}\subseteq\cL\fg$ be the locally closed ind-subscheme, denote in \cite{BKV} by
$\fC_{\bullet}$, and let $\cS_{\fg,\bullet}\in\cD([\fC_{\fg,\bullet}/\cL G])$ be the affine Grothendieck--Springer sheaf for Lie algebras denoted by $\cS_{\bullet}$ in \cite{BKV}. Then, by \cite[Theorem~7.1.4]{BKV}, $\cS_{\fg,\bullet}$ is perverse and is equipped with a $\wt{W}$-action. Hence, as in \rco{action2}(b),  $\cS_{\fg,\bullet}$ has a natural lift to an object of $\cD^{\wt{W}}([\fC_{\fg,\bullet}/\cL G])$, so for every representation $\tau\in\on{Rep}_{\bql}(\wt{W})$, one can form the sheaf of $\tau$-coinvariants $\cS_{\fg,\bullet,\tau}=\on{coinv}_{\tau}(\cS_{\fg,\bullet})\in\cD([\kC_{\fg,\bullet}/\cL G])$.
 \end{Emp}

 \begin{Thm}\label{T:pervtaulie}
For every $\tau\in\on{Rep}_{\bql}(\wt{W})$, the sheaf $\cS_{\fg,\bullet,\tau}\in\cD([\kC_{\fg,\bullet}/\cL G])$ is perverse.
\end{Thm}

\begin{proof}
To prove the result, we repeat the argument of \rt{perv-tau} word-by-word, replacing results of the current work by those of \cite{BKV} and replacing \cite[Theorem~2.3.4]{BV} by \cite[Theorem~2]{Yun}.
\end{proof}

\end{document}